%% file: main.tex
\documentclass{article}
\usepackage{xcolor}

\usepackage{amsmath}
\usepackage{amsthm, amssymb}

\newtheorem{theorem}{Theorem}[section]

\newtheorem{proposition}[theorem]{Proposition}
\newtheorem{lemma}[theorem]{Lemma}
\newtheorem{definition}[theorem]{Definition}

\usepackage{hyperref}
\newcommand{\footremember}[2]{%
	\footnote{#2}
	\newcounter{#1}
	\setcounter{#1}{\value{footnote}}%
}
\newcommand{\footrecall}[1]{%
	\footnotemark[\value{#1}]%
}

\input{header.tex}

\title{New challenges in covariance estimation: multiple structures and coarse quantization}
\author{ Johannes Maly\footnote{Catholic University of Eichstaett-Ingolstadt, Ostenstraße 26-28, 85072 Eichstätt, Germany}%
	\and Tianyu Yang\footremember{TU_Berlin}{Technical University of Berlin, Einsteinufer 25, 10587 Berlin, Germany}%
	\and Sjoerd Dirksen\footnote{Utrecht University, Budapestlaan 6, 3584 CD Utrecht, Netherlands}%
	\and Holger Rauhut\footnote{RWTH Aachen University, Pontdriesch 10, 52062 Aachen, Germany}%
	\and Giuseppe Caire\footrecall{TU_Berlin}}%
\date{}

\begin{document}
	
	\input{chapter.tex}

\end{document}

%% file: header.tex
\usepackage{type1cm}        
\usepackage{makeidx}         
\usepackage{graphicx}        
\usepackage{multicol}        
\usepackage[bottom]{footmisc}

\usepackage{newtxtext}       %
\usepackage{newtxmath}       

\newcommand{\R}{{\mathbb R}}
\newcommand{\C}{{\mathbb C}}

\newcommand{\N}{{\mathbb N}}

\DeclareMathOperator{\supp}{supp}     

\makeindex             
\usepackage{bibentry}
\nobibliography* 
\usepackage[resetlabels]{multibib}
\newcites{JM}{References}     

%% file: chapter.tex
%
%
\maketitle

\begin{abstract}
	In this self-contained chapter, we revisit a fundamental problem of multivariate statistics: estimating covariance matrices from finitely many independent samples. Based on massive Multiple-Input Multiple-Output (MIMO) systems we illustrate the necessity of leveraging structure and considering quantization of samples when estimating covariance matrices in practice. We then provide a selective survey of theoretical advances of the last decade focusing on the estimation of structured covariance matrices. This review is spiced up by some yet unpublished insights on how to benefit from combined structural constraints. Finally, we summarize the findings of our recently published preprint ``Covariance estimation under one-bit quantization'' \cite{dirksen2021covariance}  to show how guaranteed covariance estimation is possible even under coarse quantization of the samples.
\end{abstract}


\section{Introduction}
\label{JMsec:1}

The key objective in covariance estimation is simple to state. Given $n \in \N$ i.i.d.\ samples $\mathbf{X}^1,...,\mathbf{X}^n \overset{\mathrm{d}}{\sim} \mathbf{X}$ of a random vector $\mathbf{X} \in \R^p$, compute a reliable estimate of the covariance matrix $\mathbb E [\mathbf{X}\mathbf{X}^\top] = \boldsymbol{\Sigma} \in \R^{p\times p}$ (without loss of generality, we restrict ourselves here to mean-zero distributions, i.e., $\mathbb{E} [\mathbf{X}] = \boldsymbol{0}$).  For this purpose, a natural estimator is the \emph{sample covariance matrix}
\begin{align} \label{JMeq:SampleMean}
	\hat{\boldsymbol{\Sigma}}_n = \frac{1}{n} \sum_{k=1}^n \mathbf{X}^k (\mathbf{X}^k)^\top
\end{align}
as it converges to $\boldsymbol\Sigma$, for $n \rightarrow \infty$, by the law of large numbers. Nevertheless, an asymptotic result is of limited use from practical perspective. Given $n \in \N$ it provides no information on the reconstruction error $\| \hat{\boldsymbol\Sigma}_n - \boldsymbol\Sigma \|$ measured in an appropriate norm. (We will concentrate in the following on operator norm bounds.)\\ 
In the last two decades, numerous works on non-asymptotic analysis of covariance estimation showed that reliable approximation of $\boldsymbol\Sigma$ by $\hat{\boldsymbol\Sigma}_n$ becomes feasible for subgaussian distributions if $n \gtrsim p$, where $a \lesssim b$ denotes $a \le C b$ for some absolute constant $C > 0$. For instance, if $\mathbf{X}$ follows a Gaussian distribution it is well-known \cite{vershynin2018high} that with probability at least $1-2e^{-t}$
\begin{align} \label{JMeq:Estimation_Gauss}
	\| \hat{\boldsymbol\Sigma}_n - \boldsymbol\Sigma \| \lesssim \| \boldsymbol\Sigma \| \left( \sqrt{\frac{p + t}{n}} + \frac{p + t}{n} \right).
\end{align}
This classical result exhibits various weaknesses. For instance, it requires strong concentration of the distribution of $\mathbf{X}$ around its mean. The estimator in \eqref{JMeq:SampleMean} is sensitive to outliers and not reliable if concentration fails \cite{catoni2012challenging,ke2019user}. Furthermore, in applications the ambient dimension can easily exceed the number of accessible samples such that even if concentration may be assumed, the estimate in \eqref{JMeq:Estimation_Gauss} is void. 


\subsection{Outline and Notation}

In Section \ref{JMsec:MassiveMIMO} we detail Massive MIMO as one specific modern application of covariance estimation and present recent approaches from an engineering perspective. The Massive MIMO setting originates from wireless communications research and will serve as a motivation for investigating multiple structures and quantized samples in a mathematical framework. Section \ref{JMsec:StructuredCE} then surveys recent theoretical advances on estimation of structured covariance matrices and Section \ref{JMsec:OneBitCE} shows the impact of coarse sample quantization on estimation guarantees. 

We denote $[n] = \{ 1,...,n \}$. For any absolute constant $C > 0$, we abbreviate $a \le Cb$ (resp.\ $\ge$) as $a \lesssim b$ (resp.\ $\gtrsim$). Whenever we use absolute constants $c,C > 0$, their values may vary from line to line. Scalar-valued functions act component-wise on vectors and matrices. For a set $S$ the \emph{indicator function} $\chi_S$ is $1$ on $S$ and $0$ on its complement $S^c$. We denote the one-matrix by $\boldsymbol{1} \in \R^{p\times p}$ and the identity by $\mathbf I \in \R^{p\times p}$. In particular,
\begin{align*}{}
	[\mathrm{sign}(\mathbf{x})]_i = \begin{cases}
		1 & \text{if } x_i\geq 0 \\
		-1& \text{if } x_i<0,
	\end{cases}
\end{align*}
for all $\mathbf x \in \R^p$ and $i\in [p]$. For $\mathbf Z \in \R^{p\times p}$, we denote the \emph{operator norm} (maximum singular value) by $\| \mathbf Z \| = \sup_{\mathbf u \in \mathbb{S}^{p-1}} \| \mathbf Z\mathbf u \|_2$, the \emph{nuclear norm} (sum of singular values) by $\| \mathbf Z \|_* = \mathrm{tr}(\sqrt{\mathbf Z^\top \mathbf Z})$, the \emph{Frobenius norm} (trace norm) by $\| \mathbf Z \|_F^2 = \mathrm{tr} (\mathbf Z^\top \mathbf Z) = \sum_{i,j = 1}^p Z_{i,j}^2$, the \emph{max norm} by $\| \mathbf Z \|_\infty = \max_{i,j} | Z_{i,j} |$, and the \emph{maximum column norm} $\| \mathbf Z \|_{1\rightarrow 2} = \max_{j \in [p]} \| \mathbf z_j \|_2$ where $\mathbf z_j$ denotes the $j$-th column of $\mathbf Z$. We use $\odot$ for the Hadamard (i.e., entry-wise) product of two matrices. The \emph{uniform distribution} on a set $S$ is denoted by $\mathrm{Unif}(S)$. The \emph{multivariate Gaussian distribution} with mean $\boldsymbol \mu \in \R^{p}$ and covariance matrix $\boldsymbol \Sigma \in \R^{p\times p}$ is denoted by $\mathcal{N}(\boldsymbol \mu, \boldsymbol \Sigma)$. The \emph{subgaussian ($\psi_2$-) and subexponential ($\psi_1$-) norms} of a random variable $X$ are defined by 
\begin{align*}
	\| X \|_{\psi_\alpha} = \inf \left\{ t>0 \colon \mathbb E \left[ \exp\left(\tfrac{|X|^\alpha}{t^\alpha} \right) \right] \le 2 \right\}
\end{align*}
A mean-zero random vector $\mathbf X$ on $\R^n$ is called \emph{$K$-subgaussian} if 
$$\|\langle \mathbf X,\mathbf x\rangle\|_{\psi_2} \leq K \; \mathbb E [\langle \mathbf X,\mathbf x \rangle^2 ]^{1/2} \quad \mbox{ for all }  \mathbf x \in \R^n.$$


\section{Motivation --- Massive MIMO} 
\label{JMsec:MassiveMIMO}

\emph{Multiple-Input-Multiple-Output (MIMO)} is a method in wireless communication to enhance the capacity of a radio link by using multiple transmission and multiple receiving antennas. It has become an essential element of wireless communication standards for Wi-Fi and mobile devices \cite{goldsmith2003capacity,paulraj2004overview}. \emph{Massive MIMO} equips the \emph{base station (BS)} with a large number of antennas to further increase bandwidth and potential number of users \cite{lu2014overview,marzetta2016fundamentals}.

\begin{figure}[t]
	\centering
	\includegraphics[width=7.2cm]{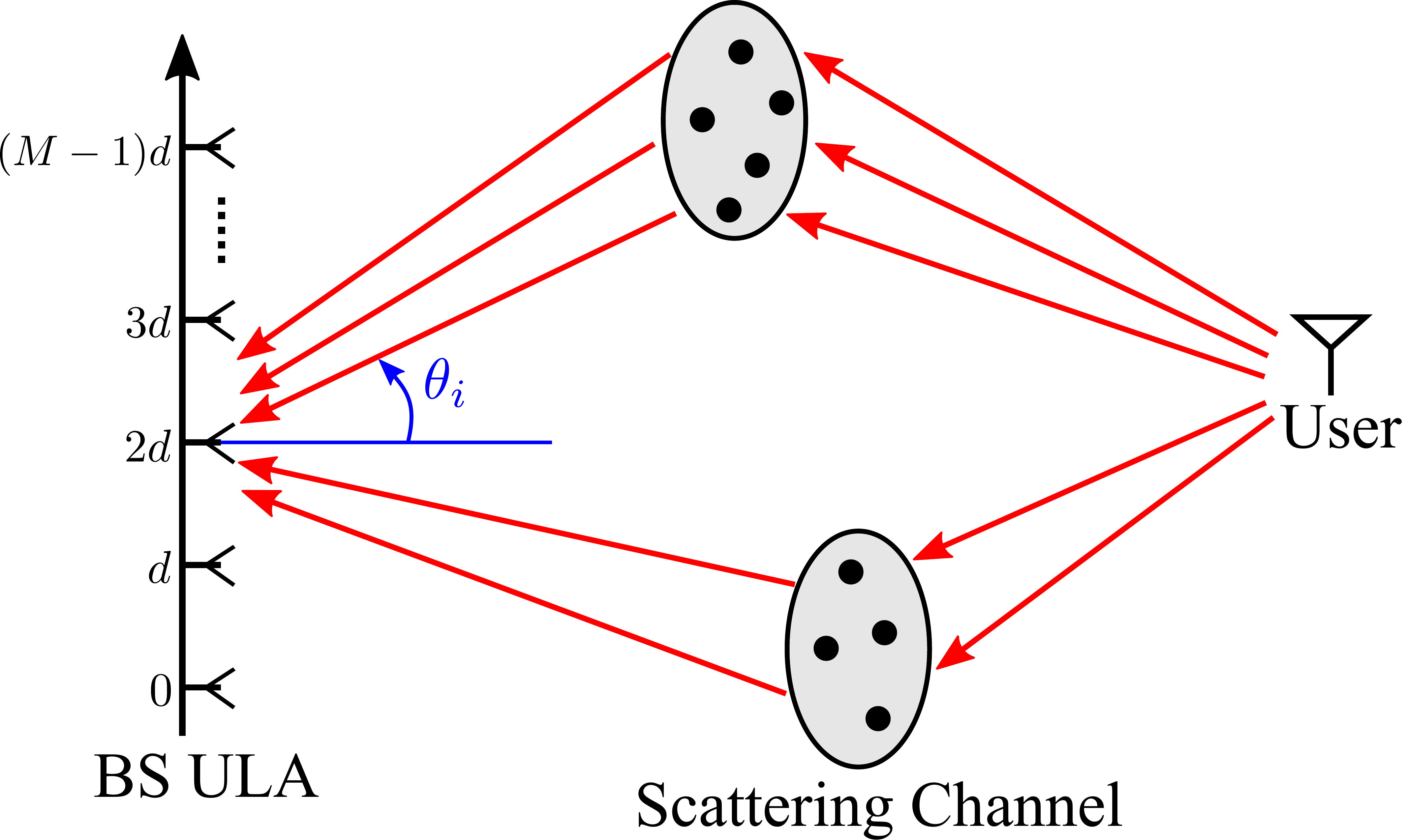}
	%
	%
	\caption{An exemplary multipath propagation channel, where the user signal is received at the BS through two scattering clusters.}
	\label{JMfig:MassiveMIMO}       
\end{figure}

We consider here a classical massive MIMO communication system, where the BS is equipped with a \emph{uniform linear array (ULA)} of $M$ antennas and communicates with multiple users through a scattering channel, e.g., wave reflection on buildings or objects. See Figure \ref{JMfig:MassiveMIMO} for an exemplary setup. During \emph{uplink (UL)} the BS receives user pilots and aims at estimating the respective channel covariance matrices, which characterize each transmission channel. By assuming mutual orthogonality of all UL pilots, it suffices to focus on a single user channel. We denote the corresponding UL channel vector at time-frequency resource $s$ by $\mathbf h(s) \in \mathbb C^M$ (standard block-fading model, e.g., \cite{tse2005fundamentals}). Furthermore, we assume that the user transmits a single pilot per channel coherence block such that the channel vectors $\mathbf h(s)$ are i.i.d., for $s \in [N]$ \cite{haghighatshoar2018low,haghighatshoar2016massive}. (To stay coherent with engineering literature, we use the therein common notation in this section. Note that our initial theoretical setting is retrieved by identifying the ambient dimension $p$ with the number of antennas $M$, the number of samples $n$ with the number of independent time-frequency resources $N$, and the sample vectors $\mathbf X^k$ with the channel vectors $\mathbf h(s)$). 

Under the above assumptions $\mathbf h(s)$ can be written as
\begin{align*}
	\mathbf h(s) = \int_{-1}^1 \rho(\xi,s) \, \mathbf a(\xi) \; \mathrm d \xi,
\end{align*}
for $s \in [N]$. Here, $\xi = \tfrac{\sin(\theta)}{\sin(\theta_{\text{max}})}$ are the normalized \emph{angles of arrival (AoA)} with $\theta_{\text{max}} \in [0,\tfrac{\pi}{2}]$ being the maximum array angular aperture, the vectors $\mathbf a(\xi) \in \mathbb C^M$ denote the respective array response at the BS antennas, and the \emph{channel gain} $\rho(\xi,s)$ is a complex Gaussian process with zero mean. By assuming the antenna spacing to be $d = \tfrac{\lambda}{2}$, where $\lambda = \tfrac{c_0}{f_0}$ denotes the wavelength with $c_0$ being the speed of light and $f_0$ the carrier frequency, we obtain that
\begin{align*}
	\mathbf a(\xi) = \begin{pmatrix}
		1, e^{j\pi \xi}, \dots, e^{j\pi (M-1) \xi}
	\end{pmatrix}^\top,
\end{align*}
where $j$ denotes the imaginary unit. With the additional assumption of \emph{wide sense stationary uncorrelated scattering (WSSUS)}, the second order statistics of the Gaussian process $\rho(\xi,s)$ is time invariant and uncorrelated across AoAs such that
\begin{align*}
	\mathbb E [ \rho(\xi,s) \rho^*(\xi',s) ] = \gamma (\xi) \, \delta (\xi - \xi'),
\end{align*}
where $\gamma \colon [-1,1] \to \R_{\ge 0}$ is the real and non-negative measure that represents the \emph{angular scattering function (ASF)} and $\delta$ is the Dirac delta function.

The received pilot signal at the BS at resource block $s$ is thus given as
\begin{align*}
	\mathbf y(s) = \mathbf h(s) x(s) + \mathbf z(s),
\end{align*}
for $s \in [N]$, where $x(s)$ is the pilot symbol and $\mathbf z(s) \sim \mathcal C \mathcal N (\boldsymbol 0, N_0 \mathbf I) = \mathcal N (\boldsymbol 0, \tfrac{N_0}{2} \mathbf I) + j \mathcal N (\boldsymbol 0, \tfrac{N_0}{2} \mathbf I)$ models \emph{additive white Gaussian noise (AWGN)}. Without loss of generality one may assume that the pilot symbols are normalized, i.e., $x(s) = 1$. The core problem of massive MIMO channel estimation is now to estimate the channel covariance matrix
\begin{align} \label{JMeq:SigmaH}
	\boldsymbol \Sigma_{\mathbf h}
	= \mathbb E [ \mathbf h(s) \mathbf h(s)^{\mathsf H} ]
	= \int_{-1}^1 \gamma (\xi) \, \mathbf a (\xi) \mathbf a (\xi)^{\mathsf H} \; \mathrm d \xi
\end{align}
from $N$ noisy samples $\mathbf y(s)$, $s\in [N]$. Since the number of samples $N$ is limited due to time constraints of the UL phase, one expects for massive MIMO that $M \approx N$, i.e., $p \approx n$. In light of \eqref{JMeq:Estimation_Gauss}, the sample covariance matrix will thus not provide a reliable estimate of $\boldsymbol \Sigma_{\mathbf h}$ in this case.

\paragraph{A hands-on approach.} In \cite{Khalilsarai2020structured} and related ongoing work, we use a more refined approach to estimate $\boldsymbol \Sigma_{\mathbf h}$. First note that by \eqref{JMeq:SigmaH} the channel covariance matrix belongs to the set
\begin{align*}
	\mathcal M = \left\{ \int_{-1}^1 \gamma(\xi) \, \mathbf a (\xi) \mathbf a (\xi)^{\mathsf H} \; \mathrm d \xi \colon \gamma \in \mathcal A \right\},
\end{align*} 
where $\mathcal A$ denotes the class of typical ASFs in wireless propagation. If one assumes sparse scattering propagation, the set $\mathcal A$ consists of sparse ASFs. In particular, we assume that $\gamma(\xi)$ can be decomposed as the sum of a discrete spike component $\gamma_d$ (modeling the power received from \emph{line of sight (LOS)} paths and narrow scatterers) and a continuous component $\gamma_c$ (modeling the power received from wide scatterers). Mathematically, we can write
\begin{align} \label{JMeq:GammaDecomp}
	\gamma (\xi) 
	= \gamma_d(\xi) + \gamma_c(\xi)
	= \sum_{k=1}^r c_k \delta(\xi - \xi_k) + \gamma_c(\xi),
\end{align}
where $\gamma_d$ consists of $r \ll M$ Dirac deltas with AoAs $\xi_1,\dots,\xi_r$ and  strengths $c_1,\dots,c_r > 0$ corresponding to $r$ specular propagation elements. Furthermore, by sparsity assumptions on $\gamma$ we have that $\mathrm{meas}(\gamma_c) \ll \mathrm{meas}([-1,1])$, where $\mathrm{meas}(\gamma_c)$ denotes here the measure of the support of $\gamma_c$. Combining \eqref{JMeq:SigmaH} and \eqref{JMeq:GammaDecomp}, we decompose the channel covariance matrix as
\begin{align} \label{JMeq:SigmaDecomp}
	\boldsymbol \Sigma_{\mathbf h} 
	= \boldsymbol \Sigma_{\mathbf h}^d + \boldsymbol \Sigma_{\mathbf h}^c
	= \sum_{k=1}^r c_k \, \mathbf a (\xi_k) \mathbf a (\xi_k)^{\mathsf H} + \int_{-1}^1 \gamma_c(\xi) \, \mathbf a (\xi) \mathbf a (\xi)^{\mathsf H} \; \mathrm d \xi,
\end{align}
where $\boldsymbol \Sigma_{\mathbf h}^d$ is rank-$r$ and positive semi-definite and $\boldsymbol \Sigma_{\mathbf h}^c$ is full rank and positive semi-definite with few dominant singular values. We can approximate $\boldsymbol \Sigma_{\mathbf h}$ now in three consecutive steps:

\begin{enumerate}
	\item[(i)] \textit{Spike Location Estimation for $\gamma_d$:} Applying the \emph{MUltiple SIgnal Classification (MUSIC)} algorithm \cite{stoica2005spectral} we estimate the AoAs $\xi_k$ of the spike component $\gamma_d$ from the noisy samples $\mathbf y(1),\dots,\mathbf y(N)$, cf.\ \cite[Theorem 1]{Khalilsarai2020structured}. Since this step is fairly standard we do not discuss the details here but refer the interested reader to \cite{Khalilsarai2020structured}. Let us only mention that the number of spikes is estimated by the number of dominant eigenvalues of $\boldsymbol \Sigma_{\mathbf y} := \mathbb E [ \mathbf y(s) \mathbf y(s)^{\mathsf H} ]$ (where one can naturally assume a corresponding gap in the spectrum since the power received via LOS paths in $\gamma_d$ dominates the power received from wide scatterers in $\gamma_c$).
	As a result, we obtain estimated spike locations $\hat\xi_k$, for $k \in [\hat r]$, and define an approximation of $\gamma_d$
	\begin{align*}
		\tilde\gamma_d(\xi) = \sum_{k=1}^{\hat r} \tilde c_k \, \delta(\xi - \hat\xi_k),
	\end{align*}
    where the coefficients $\tilde c_1,\dots, \tilde c_{\hat r} \ge 0$ still need to be estimated.
	\item[(ii)] \textit{Sparse Dictionary-Based Method:} We approximate the continuous component $\gamma_c$ over a finite dictionary of densities $\mathcal G_c := \{ \psi_i \colon [-1,1] \to \R, \; i \in [G] \}$ that are suitably chosen, e.g., Gaussian, Laplacian, or rectangular kernels, cf.\ Figure \ref{JMfig:Dictionary}. We hence define
	\begin{align*}
		\tilde\gamma_c(\xi) = \sum_{i=1}^{G} \tilde b_i \psi_i(\xi),
	\end{align*}
    where only the coefficients $\tilde b_1,\dots, \tilde b_G \ge 0$ need to be estimated.
    \item[(iii)] \textit{Non-Negative Least Square (NNLS) estimator:} Collecting the coefficients in a single vector $\mathbf u = (\tilde b_1,\dots,\tilde b_G,\tilde c_1,\dots, \tilde c_{\hat r})^\top \in \R_{\ge 0}^{G + \hat r}$ and recalling \eqref{JMeq:SigmaDecomp}, we define our coefficient dependent estimate of the channel covariance 
    \begin{align} \label{JMeq:Reduction}
    	\boldsymbol \Sigma_{\mathbf h} (\mathbf u)
    	= \sum_{k=1}^{\hat r} \tilde c_k \, \mathbf a (\hat\xi_k) \mathbf a (\hat\xi_k)^{\mathsf H} + \sum_{i=1}^G \tilde b_i \int_{-1}^1 \psi_i(\xi) \, \mathbf a (\xi) \mathbf a (\xi)^{\mathsf H} \; \mathrm d \xi
    	=: \sum_{i=1}^{G+\hat r} u_i \mathbf S_i,
    \end{align}
    where 
    \begin{align*}
    	\mathbf S_i = \begin{cases}
    		\int_{-1}^1 \psi_i(\xi) \, \mathbf a (\xi) \mathbf a (\xi)^{\mathsf H} \; \mathrm d \xi & \text{if } 1 \le i \le G \\
    		\mathbf a (\hat\xi_k) \mathbf a (\hat\xi_k)^{\mathsf H} & \text{if } G < i \le G + \hat r.
    	\end{cases}
    \end{align*}
    All that remains is to determine the coefficient vector $\mathbf u$. Since $\boldsymbol \Sigma_{\mathbf y} = \boldsymbol \Sigma_{\mathbf h}  + N_0 \mathbf I$, we can do so by fitting \eqref{JMeq:Reduction} to the sample covariance matrix $\hat {\boldsymbol \Sigma}_{\mathbf y}$ of $\mathbf y(1),\dots,\mathbf y(N)$, i.e.,
    \begin{align} \label{JMeq:NNLSgeneral}
    	\mathbf u^* 
    	= \mathrm{arg}\min_{\mathbf u \ge \boldsymbol 0} \Big\| \hat {\boldsymbol \Sigma}_{\mathbf y} - \sum_{i=1}^{G+\hat r} u_i \mathbf S_i - N_0 \mathbf I \Big\|_F^2.
    \end{align}
    Since $\boldsymbol \Sigma_{\mathbf h}$ is Hermitian Toeplitz, one can incorporate the structure in \eqref{JMeq:NNLSgeneral} by replacing $\hat{ \boldsymbol \Sigma }_{\mathbf h} = \hat{\boldsymbol \Sigma}_{\mathbf y} - N_0 \mathbf I$ with its projection $\tilde{ \boldsymbol \Sigma }_{\mathbf h}$ onto the space of Hermitian Toeplitz matrices (which can be done by averaging the diagonals, cf.\ Section \ref{JMsec:Toeplitz}). Denoting the first column of $\tilde{ \boldsymbol \Sigma }_{\mathbf h}$ by $\tilde{ \boldsymbol \sigma } \in \C^M$ and collecting the first columns of the matrices $\mathbf S_i$ in a matrix $\tilde{ \mathbf S } \in \C^{M \times (G+\hat r)}$, we may instead solve
    \begin{align} \label{JMeq:NNLS}
        \mathbf u^* 
        = \mathrm{arg}\min_{\mathbf u \ge \boldsymbol 0} \Big\| \mathbf W ( \tilde{\mathbf S} \mathbf u  - \tilde{ \boldsymbol \sigma }) \Big\|_F^2,
    \end{align}
    where $\mathbf W = \mathrm{diag} \big( (\sqrt{M}, \sqrt{2(M-1)},\sqrt{2(M-2)},...,\sqrt{2})^\top \big)$ is a weight matrix compensating the averaging process.
\end{enumerate}

\begin{figure}[t]
	\centering
	\includegraphics[width=7.2cm]{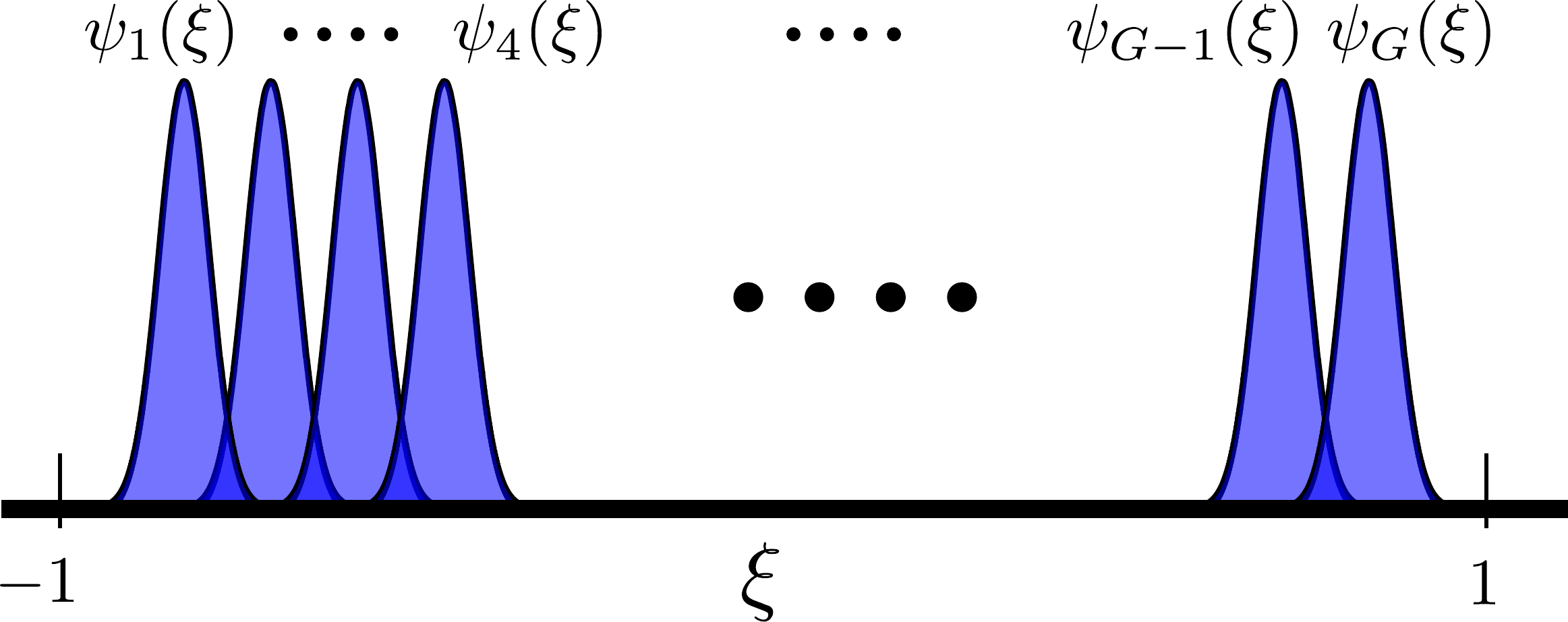}
	%
	%
	\caption{Example of a Gaussian dictionary that might be used to express $\gamma_c$.}
	\label{JMfig:Dictionary}       
\end{figure}

\paragraph{A hands-on approach --- Empirical evaluation.} Let us empirically compare the NNLS estimator to the sample covariance matrix right away. We consider a ULA with $M=128$ antennas, where the spacing between two consecutive antenna elements is set to $d = \frac{\lambda}{2}$. We produce random ASFs in the following general format:
\begin{equation}
\begin{aligned}
\gamma (\xi) &= \gamma_d (\xi) + \gamma_c (\xi) \\& = \frac{\alpha}{r}   \sum^r_{i=1}\delta (\xi - \xi_i) 
+ \frac{1-\alpha}{Z}\left( \sum^{n_r}_{j=1}\tt{rect}_{\mu_j,\sigma_j} (\xi) + \sum^{n_g}_{k=1} \tt{Gaussian}_{\mu_k,\sigma_k} (\xi) \right),
\end{aligned}
\end{equation}
where $r := 2, n_r := 2$ and $n_g := 2$ are set as the number of delta, rectangular and Gaussian functions, respectively. The spike locations are chosen uniformly at random from $[-1,1]$, i.e., $\xi_i \sim \mathrm{Unif}([-1,1])$ for $i \in [2]$. The rectangular functions are defined as
\begin{align*}
	\tt{rect}_{\mu_j,\sigma_j} (\xi) = \chi_{ \left[ \mu_j - \tfrac{\sigma_j}{2}, \mu_j + \tfrac{\sigma_j}{2} \right]} (\xi),
\end{align*}
where $\mu_1 \sim \mathrm{Unif}([-1,0])$, $\mu_2 \sim \mathrm{Unif}([0,1])$, and $\sigma_j \sim \mathrm{Unif}([0.1,0.3])$, for $j \in [2]$. The Gaussian functions $\tt{Gaussian}_{\mu_k,\sigma_k}$ are densities of $\mathcal N(\mu_k,\sigma_k)$, where $\mu_k \sim \mathrm{Unif}([-0.7,0.7])$ and $\sigma_k \sim \mathrm{Unif}([0.03,0.04])$, for $k \in [2]$. Moreover, $\alpha := 0.5$ is set to present the power contribution of discrete spikes. The constant $Z = \int_{-1}^1 \gamma_c(\xi) d\xi$ normalizes $\gamma_c$ in measure. The SNR is set to $10$ dB.

In addition to the sample covariance, we compare our NNLS estimator to sparse iterative covariance-based estimation (SPICE) \cite{stoica2011spice}. This method also exploits the ASF domain to minimize a covariance matrix fitting. Note that SPICE can only be applied with Dirac delta dictionaries and that it does not include a step of spike support detection as our method.

Denoting a generic covariance estimate as $\bar{\boldsymbol{\Sigma}}$, we consider two metrics to evaluate the estimation quality. The first metric, namely \textit{normalized Frobenius-norm error}, is defined as $E_{\text{NF}} =  \frac{\Vert \boldsymbol{\Sigma}_{\mathbf{h}} -  \bar{\boldsymbol{\Sigma}} \Vert_{\sf{F}}}{\Vert \boldsymbol{\Sigma}_{\mathbf{h}} \Vert_{\sf{F}}}$. Another metric, namely \textit{power efficiency}, evaluates the similarity of dominant subspaces between the estimated and true matrices, which is an important factor in various applications of massive MIMO such as user grouping and group-based beamforming.  Specifically, let $d\in[M]$ denote a subspace dimension parameter and let  $\mathbf{U}_d \in \mathbb{C}^{M\times d}$ and $\bar{\mathbf{U}}_d\in \mathbb{C}^{M\times d}$ be the $d$ dominant eigenvectors of $\boldsymbol{\Sigma}_{\mathbf{h}}$ and $\bar{\boldsymbol{\Sigma}}$ corresponding to their largest $d$ eigenvalues, respectively. Then, the power efficiency based on $d$ is defined as $E_{\text{PE}}(d) =  1 - \frac{\langle\boldsymbol{\Sigma}_{\mathbf{h}},\bar{\mathbf{U}}_d\bar{\mathbf{U}}_d^{\sf{H}}\rangle}{\langle\boldsymbol{\Sigma}_{\mathbf{h}},\mathbf{U}_d\mathbf{U}_d^{\sf{H}}\rangle}$. Note that $E_{\text{PE}}(d) \in [0,1]$ where a value closer to 0 means that more power is captured by the estimated  $d$-dominant subspace.  

SPICE and the proposed NNLS  estimators are  applied with $G=2M$ Dirac delta  dictionaries for the continuous part $\mathcal{G}_c$. The resulting Frobenius-norm error and power efficiency are depicted in Figure~\ref{JMfig:Simulation}. All results are averaged over 20 random ASFs and 200 random channel realizations for each ASF. The proposed NNLS method outperforms the sample covariance matrix and SPICE for both metrics. Finally, one can observe a similar outcome for smaller sample sizes as well, e.g., $N/M=0.125$, which occur naturally in massive MIMO.

\begin{figure}[t]
	\centering
	\includegraphics[width=5.9cm]{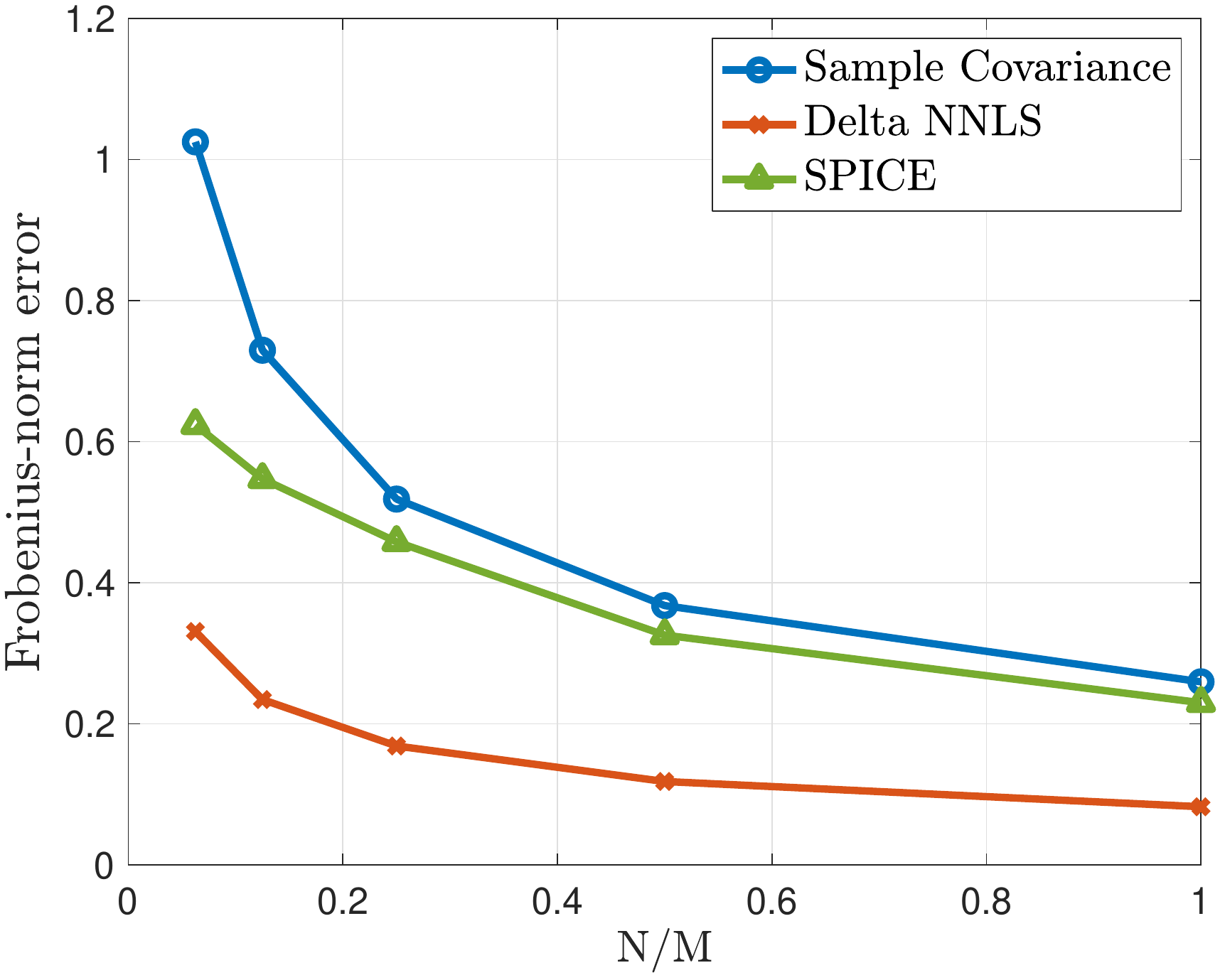}
	\;
	\includegraphics[width=5.9cm]{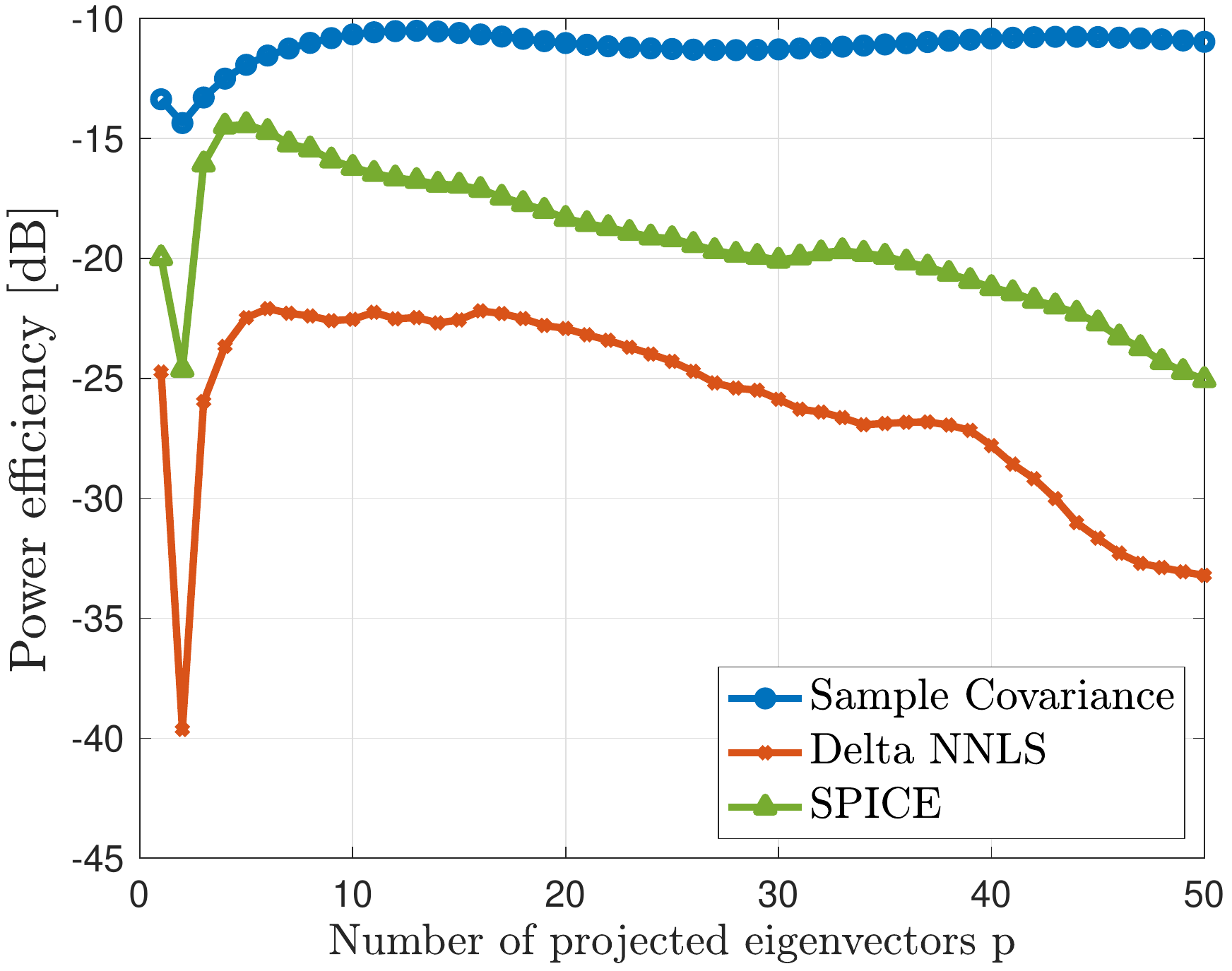}
	%
	%
	\caption{Frobenius-norm error (left) and power efficiency with $\tfrac{N}{M} = 0.5$ (right).}
	\label{JMfig:Simulation}       
\end{figure}
%

\paragraph{Structure and quantization.} Let us end this motivational section by highlighting two crucial points. First, whereas engineers are successful in boosting the sample covariance matrix by using special features of their problem setting, it might simplify existing approaches if alternatives to the sample covariance matrix are used that automatically leverage intrinsic structure(s) of the covariance matrix. As Section \ref{JMsec:StructuredCE} will show, the last decade substantially improved our theoretical understanding in this regard. Second, if the above methods are used in real applications, one has to take into account that the sample vectors $\mathbf y(s)$ have to be quantized to finite alphabets before digital processing. Especially in massive MIMO the information loss due to quantization can be significant since fine quantization at a multitude of antennas leads to enormous energy consumption. The results presented in Section \ref{JMsec:OneBitCE} can be seen as a first theoretical step into  understanding the non-asymptotic behavior of covariance estimators under coarse quantization of the samples. Since we concentrate on memoryless quantization schemes (each vector entry is quantized independently of all others), our model should be applicable to massive MIMO in a straight-forward way.


\section{Estimation of structured covariance matrices and robustness against outliers}
\label{JMsec:StructuredCE}

As we already have seen in Section \ref{JMsec:MassiveMIMO}, there are several structures of interest that $\boldsymbol \Sigma$ might exhibit in applications. We concentrate here on three important instances --- sparsity, low-rankness, and Toeplitz-structure --- that naturally emerge in engineering, biology, and data science, e.g.,  \cite{romero2016compressive,Liu13167}. Parts of the results we review below are not restricted to Gaussian random vectors but allow to treat heavy-tailed distributions that only satisfy assumptions on their lower moments. Techniques for robust covariance estimation include median-of-means \cite{nemirovskij1983problem,jerrum1986random}, element- and spectrum-wise truncation \cite{catoni2012challenging,minsker2018sub}, and $M$-estimators \cite{minsker2018sub,minsker2020robust}. The recent work \cite{mendelson2020robust} even constructs a "sub-Gaussian" estimator that only requires a finite kurtosis assumption ($L_4$-$L_2$-norm equivalence). In this context, sub-Gaussian means that the estimator performs as well as the sample covariance matrix applied to Gaussian distributions, for further discussion see \cite{mendelson2020robust}. Although the proposed construction is computationally intractable, it illustrates the potential of robust estimation. For further information on early and recent approaches to robust covariance estimation, we refer the reader to \cite{hubert2008high,ke2019user}.


\subsection{Sparse covariance matrices} 
\label{JMsec:SparseCE}

We begin with the assumption that $\boldsymbol \Sigma$ is a \emph{sparse} matrix, i.e., only few entries of $\boldsymbol \Sigma$ are relevant and hence non-zero. If $\mathbf X$ models ordered variables, the non-zero entries of $\boldsymbol \Sigma$, for instance, might cluster around the diagonal such that $\boldsymbol \Sigma$ is a banded or tapered matrix. A straight-forward way to estimate such covariance matrices is to band/taper the sample covariance matrix $\hat{\boldsymbol \Sigma}_n$ \cite{bickel2008regularized,furrer2007estimation,cai2010optimal}. If the variables are not ordered and the non-zero entries of $\boldsymbol \Sigma$ do not cluster, thresholding of $\hat{\boldsymbol \Sigma}_n$ is a viable alternative \cite{bickel2008covariance,el2008operator}. As remarked in \cite{levina2012partial}, the just named approaches can be treated in a unified way by introducing a \emph{mask} $\mathbf M \in [0,1]^{p\times p}$ and considering the \emph{masked sample covariance matrix} $\mathbf M \odot \hat{\boldsymbol \Sigma}_n$. The masked formulation allows to decompose the estimation error
\begin{align*}
	\| \mathbf M \odot \hat{\boldsymbol \Sigma }_n - \boldsymbol \Sigma \| 
	\le \| \mathbf M \odot \hat{\boldsymbol \Sigma }_n - \mathbf M \odot \boldsymbol \Sigma \| + \| \mathbf M \odot \boldsymbol \Sigma - \boldsymbol \Sigma \|
\end{align*}
into a variance term that behaves well if $\mathbf M$ is (close to) sparse and a bias term that is small whenever $\mathbf M$ encodes the support of $\boldsymbol \Sigma$. The bias term is deterministic and solely depends on a proper choice of $\mathbf M$. For understanding the influence of sparsity on the required sample size it thus suffices to control the variance term. The corresponding state-of-the-art result can be found in \cite{chen2012masked} which extends \cite{levina2012partial} from Gaussian distributions to general distributions of finite fourth moment and strengthens \cite{levina2012partial} if applied to Gaussian distributions. To facilitate the comparison with \eqref{JMeq:Estimation_Gauss}, we present the result only in the Gaussian case.

\begin{theorem}[{\cite[Theorem 1.1]{chen2012masked}}] \label{JMthm:Chen}
	Let $\mathbf M \in [0,1]^{p\times p}$, for $p \ge 3$, be fixed and $\mathbf X \sim \mathcal{N} (\boldsymbol 0, \boldsymbol \Sigma)$, for $\boldsymbol \Sigma \in \R^{p\times p}$. Then,
	\begin{align*}
		&\mathbb E \left[ \| \mathbf M \odot \hat{\boldsymbol \Sigma }_n - \mathbf M \odot \boldsymbol \Sigma \|^2 \right]^\frac{1}{2} \\
		&\lesssim \| \boldsymbol \Sigma \|
		\left( 
		\sqrt{ \frac{\| \boldsymbol \Sigma \|_\infty}{\| \boldsymbol \Sigma \|} \cdot \frac{\| \boldsymbol M \|_{1\rightarrow 2}^2 \log(p) }{ n } }
		+ \frac{\| \boldsymbol \Sigma \|_\infty}{\| \boldsymbol \Sigma \|} \cdot \frac{\| \boldsymbol M \| \log(p) \log(np) }{ n }.
		\right)
	\end{align*}
\end{theorem}

Theorem \ref{JMthm:Chen} only bounds the second moment of the variance term, but Markov's inequality can be used to obtain according estimates that hold with high probability. Furthermore, the same proof techniques apply to higher moments of the variance term as well such that exponential tail bounds can be achieved for Gaussian $\mathbf{X}$, cf.\ \cite[Section 3.3]{chen2012masked}.\\
Let us compare Theorem \ref{JMthm:Chen} with \eqref{JMeq:Estimation_Gauss}. For general covariance estimation, i.e., $\mathbf M = \boldsymbol 1$, we have $\| \boldsymbol M \|_{1\rightarrow 2}^2 = \| \boldsymbol M \| = p$ which implies that up to $\log$-factors both results are of the same order $\mathcal{O}(\sqrt{\tfrac{p}{n}} + \tfrac{p}{n})$. If $\mathbf M$ encodes sparsity, however, meaning that only up to $s \ll p$ columns and rows are non-zero and $\| \boldsymbol M \|_{1\rightarrow 2}^2 = \| \boldsymbol M \| = s$, the estimation error is considerably reduced when applying Theorem \ref{JMthm:Chen}. A similar error reduction occurs if $\mathbf{M} \odot \hat{\boldsymbol \Sigma }_n$ is a banded estimator of bandwidth $B$.

\paragraph{Estimation via thresholding.} While the masked framework provides a unified understanding of the intrinsic complexity of sparse covariance estimation, in practice the mask $\mathbf{M}$ is unknown. A more realistic approach to the problem are hence thresholding procedures as, e.g., \cite{bickel2008covariance}. To allow for non-ordered covariance matrices, i.e., general sparsity and not only limited bandwidth of the matrix, the authors of \cite{bickel2008covariance} introduce the set of bounded and (effectively) sparse covariance matrices
\begin{align*}
	\mathcal{U} (q,s,M) := \left\{ \boldsymbol \Sigma \colon \Sigma_{i,i} \le M \text{ and } \sum_{j=1}^p |\Sigma_{i,j}|^q \le s, \text{ for all } i \in [p] \right\},
\end{align*}
for $q \in [0,1)$ and $s, M > 0$. If $q = 0$, the matrices in $\mathcal{U} (q,s,M)$ have at most $s$ non-zero entries per row; if $q > 0$, the rows are close to $s$-sparse vectors. To estimate $\boldsymbol \Sigma \in \mathcal{U} (q,s,M)$, the thresholded estimator $\mathbb T_\tau (\hat{\boldsymbol \Sigma }_n)$ is considered, where
\begin{align} \label{JMeq:ThresholdingOperator}
	[\mathbb T_\tau (\mathbf A)]_{i,j} = \begin{cases}
		A_{i,j} & \text{ if } |A_{i,j}| \ge \tau, \\
		0 & \text{ else,}
	\end{cases}
\end{align} 
for any $\tau > 0$ and $\mathbf A \in \R^{p\times p}$. 

\begin{theorem}[{\cite[Theorem 1]{bickel2008covariance}}] \label{JMthm:Bickel}
	Let $\mathbf X \sim \mathcal{N} (\boldsymbol 0, \boldsymbol \Sigma)$, for $\boldsymbol \Sigma \in \mathcal{U} (q,s,M)$, and $M' > 0$ be sufficiently large (depending on $M$). If 
	\begin{align*}
		\tau = M' \sqrt{\frac{\log(p)}{n}},
	\end{align*}
    for $n \gtrsim \log(p)$, then with probability at least $1 - e^{-cn\tau^2}$
    \begin{align*}
    	\| \mathbb T_\tau (\hat{\boldsymbol \Sigma }_n) - \boldsymbol \Sigma \|
    	= \mathcal O \left( s \left( \frac{\log(p)}{n} \right)^{\frac{1-q}{2}} \right).
    \end{align*}
\end{theorem}

Theorem \ref{JMthm:Bickel} does not require knowledge on the support of $\boldsymbol \Sigma$ and respects sparsity defects. However, if we once more consider the case $q=0$, we see that the estimate in Theorem \ref{JMthm:Bickel} is sub-optimal since the error behaves (up to $\log$-factors) like $\mathcal{O} \big( \sqrt{\tfrac{s^2}{n}} \big)$ and not like $\mathcal{O} (\sqrt{\tfrac{s}{n}})$ as one would expect.


\subsection{Low-rank covariance matrices}

When working with high-dimensional random vectors, another commonly considered structural prior is to assume that the distribution concentrates around a low-dimensional manifold. This may manifest itself in $\boldsymbol \Sigma$ being a low rank matrix. Interestingly enough, the sample covariance matrix in \eqref{JMeq:SampleMean} intrinsically leverages low-rankness of $\boldsymbol \Sigma$. To understand this phenomenon, one needs the notion of effective rank. Let us define
\begin{align*}
	\mathbf r (\boldsymbol \Sigma) = \frac{\| \boldsymbol \Sigma \|_*}{\| \boldsymbol \Sigma \|}
\end{align*}
to be the \emph{effective rank} of $\boldsymbol \Sigma$. It is straight-forward to verify that $1 \le \mathbf r (\boldsymbol \Sigma) \le \mathrm{rank}(\boldsymbol \Sigma)$. In contrast to the rank of $\boldsymbol \Sigma$, the quantity $\mathbf r (\boldsymbol \Sigma)$ is small even if $\boldsymbol \Sigma$ is only close to a low-rank matrix, e.g., consider $\boldsymbol \Sigma$ to be a full rank matrix with exponentially decaying spectrum. 

\begin{theorem}[{\cite[Corollary 2]{koltchinskii2014concentration}}] \label{JMthm:Koltchinski}
	Let $\mathbf X \sim \mathcal N (\boldsymbol 0,\boldsymbol \Sigma)$, for $\boldsymbol \Sigma \in \R^{p\times p}$, and $n \gtrsim \mathbf r (\boldsymbol \Sigma)$. Then with probability at least $1 - e^{-t}$ the sample covariance matrix satisfies
	\begin{align*}
		\| \hat{ \boldsymbol \Sigma }_n - \boldsymbol \Sigma \| 
		\lesssim \| \boldsymbol \Sigma \| \left( \sqrt{\frac{\mathbf r ( \boldsymbol \Sigma )}{n}} + \frac{\mathbf r ( \boldsymbol \Sigma )}{n} + \sqrt{\frac{t}{n}} + \frac{t}{n} \right).
	\end{align*}
\end{theorem}

The authors of \cite{koltchinskii2014concentration} further show that the bound in Theorem \ref{JMthm:Koltchinski} is tight up to constants. If we compare the result to \eqref{JMeq:Estimation_Gauss}, we see that both estimates agree for (effectively) full rank matrices like $\boldsymbol \Sigma = \mathbf I$. If $\boldsymbol \Sigma$ is of low rank, however, Theorem \ref{JMthm:Koltchinski} controls the estimation error even in the case $n < p$.

\paragraph{Low-rank estimators.} We could stop at this point since $\hat{ \boldsymbol \Sigma }_n$ apparently meets our requirements. Nevertheless, two questions remain. First, if one assumes $\boldsymbol \Sigma$ to be low-rank, one would wish for a estimator that is low-rank itself and, second, Theorem \ref{JMthm:Koltchinski} fails if $\mathbf X$ does not exhibit strong concentration around its mean. The first point can be addressed by using the LASSO-estimator
\begin{align} \label{JMeq:LASSOestimator}
	\hat{ \boldsymbol \Sigma }_n^\lambda
	= \mathrm{arg}\min_{\mathbf S \succcurlyeq \boldsymbol 0} \| \mathbf S - \hat{\boldsymbol \Sigma}_n \|_F^2 + \lambda \| \mathbf S \|_* \;,
\end{align}
where $\lambda > 0$ is a tunable parameter. Initially introduced in \cite{lounici2014high} to estimate covariance matrices from incomplete observations, the result reads in our setting as follows.

\begin{theorem}[{\cite[Corollary 1]{lounici2014high}}] \label{JMthm:Lounici}
	Let $\mathbf X \sim \mathcal N (\boldsymbol 0,\boldsymbol \Sigma)$, for $\boldsymbol \Sigma \in \R^{p\times p}$, and $n \gtrsim \mathbf r (\boldsymbol \Sigma) \log(2p + n)^2$. If
	\begin{align*}
		\lambda = C \sqrt{ \mathrm{tr} (\hat{\boldsymbol \Sigma}_n) \| \hat{\boldsymbol \Sigma}_n \| } \sqrt{ \frac{\log(2p)}{n} },
	\end{align*}
    for a sufficiently large absolute constant $C > 0$, then with probability at least $1-\tfrac{1}{2p}$ the estimator in \eqref{JMeq:LASSOestimator} satisfies
    \begin{align*}
    	\| \hat{ \boldsymbol \Sigma }_n^\lambda - \boldsymbol \Sigma \|
    	\lesssim \| \boldsymbol \Sigma \| \sqrt{ \frac{\mathbf r (\boldsymbol \Sigma) \log(2p)}{n} }.
    \end{align*}
\end{theorem}

The nuclear norm regularization in \eqref{JMeq:LASSOestimator} induces (effective) low-rankness on $\hat{ \boldsymbol \Sigma }_n^\lambda$ \cite{recht2010guaranteed} and the order of estimation error reflects up to $\log$-factors the one in Theorem \ref{JMthm:Koltchinski}. Furthermore, the construction of $\hat{ \boldsymbol \Sigma }_n^\lambda $ can easily be adapted to heavy-tailed distributions by replacing $\hat{ \boldsymbol \Sigma }_n$ with an appropriate robust counterpart, e.g., the spectrum-wise truncated sample covariance matrix \cite{ke2019user}. A corresponding version of Theorem \ref{JMthm:Lounici} that is not restricted to (sub)-Gaussian distributions is \cite[Theorem 5.2]{ke2019user}.


\subsection{Toeplitz covariance matrices and combined structures}
\label{JMsec:Toeplitz}

The third structure we discuss here in detail naturally arises in various engineering problems. If the entries of $\mathbf X$ resemble measurements on a temporal or spatial grid whose covariances only depend on the distances of measurements (in time or space) but not their location, $\boldsymbol \Sigma$ is a \emph{symmetric Toeplitz matrix}, i.e.,
\begin{align*}
	\boldsymbol \Sigma = \begin{pmatrix}
		\sigma_1 & \sigma_2 & \cdots & \sigma_{p} \\
		\sigma_2 & \ddots & \ddots & \vdots \\
		\vdots & \ddots & & \sigma_2 \\
		\sigma_{p} & \cdots & \sigma_2 & \sigma_1
	\end{pmatrix}
\end{align*}
and the first column $\boldsymbol \sigma \in \R^p$ determines $\boldsymbol \Sigma$ via $\Sigma_{i,j} = \sigma_{|i-j|+1}$. (For simplicity we identify Toeplitz matrices with their first column in the following.) Such a structure appears, for instance, in Direction-Of-Arrival (DOA) estimation \cite{krim1996two} and medical/radar imaging processing \cite{brookes2008optimising,snyder1989use}. For further examples, we refer the reader to \cite{romero2016compressive}. Since Toeplitz structure reduces the degrees of freedom in $\boldsymbol \Sigma$ from $p^2$ to $p$, leveraging this structure can lead to a notable reduction in sample complexity.

The authors of \cite{cai2013optimal} propose to average the sample covariance matrix along its diagonals to obtain the Toeplitz estimator $\hat{ \boldsymbol \Sigma }_n^\text{Toep}$ defined as
\begin{align*}
	[\hat{\sigma}_n^\text{Toep}]_r = \frac{1}{(p+1)-r} \sum_{i-j = r-1} [\hat{\Sigma}_n]_{i,j},
	\quad \text{for } r \in [p].
\end{align*}
They derive error estimates for Gaussian distributions with banded Toeplitz covariance matrices.

The more recent work \cite{kabanava2017masked} extends these results to non-Gaussian distributions and general masks as introduced in Section \ref{JMsec:SparseCE}. To be more precise, the authors of \cite{kabanava2017masked} assume that the distribution of $\mathbf X$ has the so-called convex concentration property.
\begin{definition}
	A random vector $\mathbf X \in \R^p$ has the \emph{convex concentration property} with constant $K$ if for any $1$-Lipschitz function $\phi \colon \R^p \to \R$, one has $\mathbb E [\phi(\mathbf X)] < \infty$ and 
	\begin{align*}
		\mathrm{Pr} \left[ | \phi(\mathbf X) - \mathbb E [\phi(\mathbf X)] | \ge t \right] \le 2 e^{-\tfrac{t^2}{K^2}}, 
		\quad \text{for all } t > 0.
	\end{align*}
\end{definition}
By setting $\phi = \mathrm{Id}$ one easily sees that all distributions which have the convex concentration property are subgaussian. For the sake of consistency we hence restrict ourselves here to Gaussian distributions as their most prominent representative. For a symmetric Toeplitz mask $\mathbf M \in [0,1]^{p\times p}$ characterized by its first column $\mathbf m \in [0,1]^p$, we furthermore define the weighted $\ell_1$- and $\ell_2$-norms of $\mathbf m$ as
\begin{align*}
	\| \mathbf m \|_{1,*} = \sum_{r = 1}^p \frac{m_r}{(p+1)-r}
	\quad \text{and} \quad 
	\| \mathbf m \|_{2,*} = \left( \sum_{r = 1}^p \frac{m_r^2}{(p+1)-r} \right)^\frac{1}{2}.
\end{align*} 

\begin{theorem}[{\cite[Theorem 3]{kabanava2017masked}}] \label{JMthm:Kabanava}
	Let $\mathbf M \in [0,1]^{p\times p}$ be a symmetric Toeplitz mask and $\mathbf X \sim \mathcal N(\boldsymbol 0,\boldsymbol \Sigma)$, for $\boldsymbol \Sigma \in \R^{p\times p}$ symmetric and Toeplitz. Then,
	\begin{align*}
		\mathbb E [ \| \mathbf M \odot \hat{\boldsymbol \Sigma}_n^\text{Toep} - \mathbf M \odot \boldsymbol \Sigma \| ]
		\lesssim 
		\| \boldsymbol \Sigma \| \left( \sqrt{\frac{ \| \mathbf m \|_{2,*} \log(p) }{n}} + \frac{ \| \mathbf m \|_{1,*} \log(p) }{n} \right).
	\end{align*}
\end{theorem}

As Theorem \ref{JMthm:Chen}, the result is not restricted to an estimate of the expected error but includes respective high probability bounds with exponential tail decay. Let us compare Theorem \ref{JMthm:Kabanava} to Theorem \ref{JMthm:Chen}. If we ignore $\log$-factors and $\mathbf M$ is a banding or tapering mask with support band-width $B \le \tfrac{p}{2}$, Theorem \ref{JMthm:Kabanava} guarantees an estimation error of order $\mathcal{O}(\sqrt{\tfrac{B}{pn}} + \tfrac{B}{pn})$, cf.\ \cite[Corollary 2]{kabanava2017masked}, which improves the estimate $\mathcal{O}(\sqrt{\tfrac{B}{n}} + \tfrac{B}{n})$ of Theorem \ref{JMthm:Chen} by a factor $p$. This improvement corresponds to the reduction in degrees of freedom when comparing Toeplitz to general matrices. Note, however, that the additional assumption $B \le \alpha p$, for $\alpha \in (0,1)$, is required for such a reduction since estimation of the outermost diagonals of $\boldsymbol \Sigma$ is hardly enhanced by averaging over the Toeplitz structure. This is expressed by Theorem \ref{JMthm:Kabanava} since $\| \mathbf m \|_{1,*}$ and $\| \mathbf m \|_{2,*}$ are $\mathcal O(1)$ and not $\mathcal{O}(\tfrac{1}{p})$ if the tail entries of $\mathbf m$ are not of vanishing magnitude.

\paragraph{Estimation via thresholding.} Theorem \ref{JMthm:Kabanava} differs from the previously discussed results in the sense that it allows to simultaneously leverage two structures of $\boldsymbol \Sigma$, sparsity and Toeplitz structure. Nevertheless, as in Section \ref{JMsec:SparseCE} the masked framework leaves open the question of how to choose $\mathbf M$ in practice. By combining the thresholded approach in Theorem \ref{JMthm:Bickel} with the techniques of Theorem \ref{JMthm:Kabanava} one can obtain a thresholded Toeplitz estimator which profits from both structural priors. To state a corresponding estimate, let us define the set of bounded Toeplitz covariance matrices with (effectively) sparse first column $\boldsymbol \sigma$ by
\begin{align*}
	\mathcal{U}^\text{Toep} (q,s,M) := \left\{ \boldsymbol \Sigma \colon \Sigma_{i,j} = \sigma_{|i-j|+1} \le M, \text{ for } \boldsymbol \sigma \in \R^p \text{ with } \sum_{r=1}^p |\sigma_{r}|^q \le s \right\}.
\end{align*}
We furthermore denote by $\mathbb B_{\alpha p} (\boldsymbol \Sigma)$ the matrix $\boldsymbol \Sigma$ restricted to band-width $\alpha p$, i.e., $[\mathbb B_{\alpha p} (\boldsymbol \Sigma)]_{i,j} = \Sigma_{i,j}$ if $|i-j|+1 \le \alpha p$ and $[\mathbb B_{\alpha p} (\boldsymbol \Sigma)]_{i,j} = 0$ else.

\begin{theorem} \label{JMthm:ToeplitzThresholding}
	Let $\mathbf X$ have the convex concentration property with constant $K$. Let $\mathbb E [\mathbf X] = \boldsymbol 0$ and  $\mathbb E [\mathbf X \mathbf X^\top] = \boldsymbol \Sigma$, for $\boldsymbol \Sigma \in \mathcal{U}^\text{Toep} (q,s,M)$. There exists an absolute constant $C>0$ such that, for all $\alpha \in (0,1)$ and $c > 1$, the following holds with probability at least $1 - (2\alpha p)^{-(c-1)}$. If
	\begin{align} \label{JMeq:TauRule}
		\tau = \sqrt{\frac{2c}{(1-\alpha)}} \max \{CK^2,\sqrt{C}K\} \sqrt{\frac{\log(p)}{np}},
	\end{align}
    then
    \begin{align*}
    	\left\| \mathbb T_\tau ( \mathbb B_{\alpha p} (\hat{\boldsymbol \Sigma}_n^\text{Toep} )) - \boldsymbol \Sigma \right\| 
    	\lesssim s \left( \max \{C^2K^4,CK^2\} \frac{c}{1 -  \alpha} \frac{\log(p)}{np} \right)^{\frac{1-q}{2}} + \| \mathbb B_{\alpha p} (\boldsymbol \Sigma) - \boldsymbol \Sigma \|,
    \end{align*}
    where $\mathbb T_\tau$ is the thresholding operator from \eqref{JMeq:ThresholdingOperator}.
\end{theorem}

Two comments are in order here. To gain from the Toeplitz structure, Theorem \ref{JMthm:ToeplitzThresholding} requires $\boldsymbol \Sigma$ to be close to a banded matrix. This is as in Theorem \ref{JMthm:Kabanava} before and has been discussed previously. Moreover, by adapting the proof strategy of Theorem \ref{JMthm:Bickel} the result inherits the slightly sub-optimal error decay in the sparsity level $s$, cf.\ the discussion of Theorem \ref{JMthm:Bickel} for the case $q=0$.\\
To show Theorem \ref{JMthm:ToeplitzThresholding}, we need the following lemma. In the remaining section, $\boldsymbol \sigma$ always refers to the first column of $\boldsymbol \Sigma$ and $\hat{ \boldsymbol \sigma}$ to the first column of $\hat{\boldsymbol \Sigma}_n^\text{Toep} $.

\begin{lemma} \label{JMlem:1}
	Under the assumptions of Theorem \ref{JMthm:ToeplitzThresholding}, we have for $\alpha \in (0,1)$ and $0 < u < 1$ that
	\begin{align*}
		\mathrm{Pr} \left[ \max_{r \le \alpha p} |\hat{\sigma}_r - \sigma_r| \ge \sqrt{u} \right] \le 2 \alpha p e^{-(1-\alpha) \min\left\{ \frac{1}{CK^4},\frac{1}{CK^2} \right\} n p u},
	\end{align*}
	where $C > 0$ is an absolute constant.
\end{lemma}

\begin{proof}
	We proceed similar as in \cite{kabanava2017masked}. First note that, for all $k \in [n], r \in [\alpha p]$, we can write
	\begin{align} \label{JMeq:equivalence}
	\begin{split}
		|\hat{\sigma}_r - \sigma_r| 
		&= \frac{1}{(p+1)-r} \left| \sum_{j-i=r-1} \left( X_i^k X_j^k - \sigma_r \right) \right| \\
		&= \left| \langle \mathbf M_r \mathbf X^k,\mathbf X^k \rangle - \mathbb E [\langle \mathbf M_r \mathbf X^k,\mathbf X^k \rangle ] \right|,
	\end{split}
	\end{align}
	where the mask $\mathbf M_r$ is defined by $[M_r]_{i,j} = \tfrac{1}{(p+1)-r}$ if $j-i = r-1$ and $[M_r]_{i,j} = 0$ else, i.e., 
	only the $r$-th co-diagonal of $\mathbf M$ is non-zero.
    By using a version of the Hanson-Wright inequality for random vectors with the convex concentration property \cite{adamczak2015note}, we get that
	\begin{align*}
		\mathrm{Pr} \big[ \big| \underset{=: Z_i^r}{\underbrace{\langle \mathbf M_r \mathbf X^k,\mathbf X^k \rangle - \mathbb E [\langle \mathbf M_r \mathbf X^k,\mathbf X^k \rangle ]}} \big| \ge u \big] \le 2 e^{-\min \left\{ \frac{u^2}{CK^4 \| \mathbf M_r \|_{F}^2}, \frac{u}{CK^2 \| \mathbf M_r \| } \right\}},
	\end{align*}
	which, by integration, leads to
	\begin{align*}
		\mathbb E [|Z_i^r|^{2q}] 
		&\le 2q (2CK^4 \| \mathbf M_r \|_{F}^2)^q \Gamma(q) + 4q (CK^2 \| \mathbf M_r \| )^{2q} \Gamma(2q) \\
		&\le q! (4CK^4 \| \mathbf M_r \|_{F}^2)^q + (2q)! (2CK^2 \| \mathbf M_r \| )^{2q},
	\end{align*}
    for any $q \ge 1$. The random variables $Z_i^r$ are thus sub-gamma with variance $\nu = 16K^4 (C\| \mathbf M_r \|_{F}^2 + C^2\| \mathbf M_r \|^2)$ and scale parameter $\gamma = 2CK^2 \| \mathbf M_r \|^2$ \cite[Theorem 2.3]{boucheron2013concentration}. By independence, we get for all $0 < \mu < \tfrac{1}{\gamma}$ 
    \begin{align*}
    	\mathbb E \left[ e^{\mu \sum_{i=1}^n Z_i^r} \right]
    	= \prod_{i=1}^{n} \mathbb E [ e^{\mu Z_i^r} ]
    	\le e^{\frac{\mu^2 n \nu}{2(1-\gamma \mu)}}
    \end{align*}
    (and the same holds for $-Z_i^r$) such that $\sum_{i=1}^{n} Z_i^r$ is sub-gamma with variance factor $\nu n$ and scale parameter $\gamma$ \cite[Chapter 2.4]{boucheron2013concentration}. Consequently, 
	\begin{align*}
		& \mathrm{Pr} \left[  \left| \frac{1}{n} \sum_{i=1}^n Z_i^r \right| \ge C K^2 \left( \| \mathbf M_r \|_{F} \sqrt{\frac{u}{n}} + \| \mathbf M_r \| \frac{u}{n} \right) \right] \\
		&\le \mathrm{Pr} \left[  \left| \frac{1}{n} \sum_{i=1}^n Z_i^r \right| \ge \sqrt{2\nu n u} + \gamma u \right]
		\le 2 e^{-u},
	\end{align*}
    for any $u > 0$ \cite[Chapter 2.4]{boucheron2013concentration}. Recalling \eqref{JMeq:equivalence} and noting that $\| \mathbf M_r \|_{F}^2 = \| \mathbf M_r \| = \frac{1}{(p+1)-r}$ yields with the choice $u = \min\big\{ \tfrac{1}{C^2K^4},\tfrac{1}{CK^2} \big\} ((p+1)-r)n \tilde{u}$ that
	\begin{align*}
		\mathrm{Pr} \left[ |\tilde{\sigma}_r - \sigma_r| \ge 2\sqrt{\tilde{u}} \right] &\le \mathrm{Pr} \left[ |\tilde{\sigma}_r - \sigma_r| \ge \sqrt{\tilde{u}} + \tilde{u} \right] \\
		&\le 2e^{-\min\left\{ \frac{1}{C^2K^4},\frac{1}{CK^2} \right\} ((p+1)-r) n \tilde{u}}.
	\end{align*}
	A union bound over $r \in [\alpha p]$ and the bound $r \le \alpha p$ conclude the proof.
\end{proof}

The proof of Theorem \ref{JMthm:ToeplitzThresholding} now follows along the lines of \cite[Theorem 1]{bickel2008covariance}.

\begin{proof}[Proof of Theorem \ref{JMthm:ToeplitzThresholding}]
	Let us assume that $\boldsymbol \Sigma$ has a bandwidth of at most $\alpha p$, i.e., $\mathbb B_{\alpha p} (\boldsymbol \Sigma) = \boldsymbol \Sigma$ and $\supp(\boldsymbol \sigma) \subset [\alpha p]$. The general claim then follows from
	\begin{align*}
		\| \mathbb T_\tau ( \mathbb B_{\alpha p} (\hat{\boldsymbol \Sigma}_n^\text{Toep} )) - \boldsymbol \Sigma \| 
		\lesssim \| \mathbb T_\tau ( \mathbb B_{\alpha p} (\hat{\boldsymbol \Sigma}_n^\text{Toep} )) - \mathbb B_{\alpha p} (\boldsymbol \Sigma) \|  + \| \mathbb B_{\alpha p} (\boldsymbol \Sigma) - \boldsymbol \Sigma \|.
	\end{align*}
	By Lemma \ref{JMlem:1}, we get with probability at least $1 - (2\alpha p)^{-(c-1)}$ that
	\begin{align} \label{JMeq:maxBound}
		\max_{r \le \alpha p} |\hat{\sigma}_r - \sigma_r| \le \sqrt{\frac{c}{1-\alpha}} \max \{CK^2,\sqrt{C}K\} \sqrt{\frac{\log(p)}{np}},
	\end{align}
	where $c > 1$. For convenience, let us abbreviate $\tilde{\boldsymbol \Sigma} := \mathbb B_{\alpha p} (\hat{\boldsymbol \Sigma}_n^\text{Toep} )$ and denote its first column by $\tilde{\boldsymbol \sigma}$. We compute
	\begin{align*}
		\| \mathbb T_\tau ( \tilde{\boldsymbol \Sigma} ) - \boldsymbol \Sigma \|  
		\le \left\| \mathbb T_\tau ( \boldsymbol \Sigma ) - \boldsymbol \Sigma \right\|  
		+ \| \mathbb T_\tau ( \tilde{\boldsymbol \Sigma} ) - \mathbb T_\tau (\boldsymbol \Sigma) \|,
	\end{align*}
	where the elementary estimate
	\begin{align*}
		\sum_{j=1}^p |\Sigma_{i,j}| \chi_{\{ |\Sigma_{i,j}| \le \tau\}} = \sum_{j=1}^p |\Sigma_{i,j}|^q |\Sigma_{i,j}|^{1-q} \chi_{\{ |\Sigma_{i,j}| \le \tau\}}
		\le \tau^{1-q} \sum_{j=1}^p |\Sigma_{i,j}|^q
	\end{align*}
	yields
	\begin{align} \label{JMeq:bound}
		\left\| \mathbb T_\tau ( \boldsymbol \Sigma ) - \boldsymbol \Sigma \right\| 
		\le \max_i \sum_{j=1}^p |\Sigma_{i,j}| \chi_{\{ |\Sigma_{i,j}| \le \tau \}} \le \tau^{1-q} s
	\end{align}
    via Gershgorin's disc theorem. Moreover,
	\begin{align*}
	    \| \mathbb T_\tau ( \tilde{\boldsymbol \Sigma} ) - \mathbb T_\tau (\boldsymbol \Sigma) \| 
	    &\le  \max_i \sum_{j=1}^p |\tilde{\Sigma}_{i,j}| \chi_{\{ |\tilde{\Sigma}_{i,j}| \ge\tau,\; |\Sigma_{i,j}| <\tau\}} \\
	    &+ \max_i \sum_{j=1}^p |\Sigma_{i,j}| \chi_{\{ |\tilde{\Sigma}_{i,j}| <\tau,\; |\Sigma_{i,j}| \ge\tau\}} \\
	    &+ \max_i \sum_{j=1}^p |\tilde{\Sigma}_{i,j} - \Sigma_{i,j}| \chi_{\{ |\tilde{\Sigma}_{i,j}| \ge\tau,\; |\Sigma_{i,j}| \ge\tau\}} \\
	    &= (I) + (II) + (III).
	\end{align*}
	First recall that by assumption $\supp(\boldsymbol \sigma) \subset [\alpha p]$ and $\supp(\tilde{\boldsymbol \sigma}) \subset [\alpha p]$. Hence, using the observation that $\tilde{\sigma}_r = \hat{\sigma}_r$, for $r \le \alpha p$, and
	\begin{align} \label{JMeq:bound2}
	    \sum_{j=1}^p \chi_{\{ |\Sigma_{i,j}| \ge \tau \}}
	    = \sum_{j=1}^p \tau^q \tau^{-q} \chi_{\{ |\Sigma_{i,j}| \ge \tau\}}
	    \le \sum_{j=1}^p |\Sigma_{i,j}|^q \tau^{-q},
	\end{align}
	we may estimate with \eqref{JMeq:maxBound} and the definition \eqref{JMeq:TauRule} of $\tau$
	\begin{align*}
	    (III) &\le \max_{r\le \alpha p} |\hat{\sigma}_r - \sigma_r| \cdot \max_i \sum_{j=1}^p |\Sigma_{i,j}|^q \tau^{-q} \\
	    &\le s \tau^{-q} \sqrt{\frac{c}{1-\alpha}} \max \{CK^2,\sqrt{C}K\} \sqrt{\frac{\log(p)}{np}} \\
	    &\le s \tau^{1-q}. 
	\end{align*}
    Furthermore,
    \begin{align*}
    	(I) &\le \max_i \sum_{j=1}^p |\tilde{\Sigma}_{i,j} - \Sigma_{i,j}| \chi_{\{ |\tilde{\Sigma}_{i,j}| \ge \tau,\; |\Sigma_{i,j}| < \tau\}} 
    	+ \max_i \sum_{j=1}^p |\Sigma_{i,j}| \chi_{\{ |\tilde{\Sigma}_{i,j}| \ge \tau,\; |\Sigma_{i,j}| < \tau\}} \\
    	&= (IV) + (V).
    \end{align*}
    By \eqref{JMeq:bound}, we know that
    \begin{align*}
    	(V) \le \tau^{1-q} s.
    \end{align*}
    Now take $\gamma \in (0,1)$. We get that
    \begin{align*}
    	(IV) &\le \max_i \sum_{j=1}^p |\tilde{\Sigma}_{i,j} - \Sigma_{i,j}| \chi_{\{ |\tilde{\Sigma}_{i,j}| \ge \tau,\; |\Sigma_{i,j}| < \gamma \tau\}} 
    	+ \max_i \sum_{j=1}^p |\tilde{\Sigma}_{i,j} - \Sigma_{i,j}| \chi_{\{ |\tilde{\Sigma}_{i,j}| \ge \tau,\; \gamma \tau \le |\Sigma_{i,j}| < \tau\}} \\
    	&\le \max_{r \le \alpha p} |\tilde{\sigma}_r - \sigma_r| \cdot \max_i N_i(1-\gamma) \\
    	&+ s (\gamma \tau)^{-q} \sqrt{\frac{c}{1-\alpha}} \max \{CK^2,\sqrt{C}K\} \sqrt{\frac{\log(p)}{np}}, 
    \end{align*}
    where we defined $N_i(1-\gamma) := \sum_{j=1}^p \chi_{\{ |\tilde{\Sigma}_{i,j} - \Sigma_{i,j}| > (1-\gamma) \tau\}}$ and re-used the bound on $(III)$ for the second term. Since we have by \eqref{JMeq:maxBound} and the definition \eqref{JMeq:TauRule} of $\tau$ that $N_i(1-\gamma) = 0$, for $i \in [p]$ and $\gamma$ with $(1-\gamma)\sqrt{2} \ge 1$,
    we get that
    \begin{align*}
    	(IV) \lesssim s \tau^{-q} \sqrt{\frac{c}{1-\alpha}} \max \{CK^2,\sqrt{C}K\} \sqrt{\frac{\log(p)}{np}}.
    \end{align*}
    Hence,
    \begin{align*}
    	(I) \lesssim s \tau^{1-q}.
    \end{align*}
    Finally, note that by \eqref{JMeq:bound2}
    \begin{align*}
    	(II) &\le \max_i \sum_{j=1}^p (|\tilde{\Sigma}_{i,j} - \Sigma_{i,j}| + |\tilde{\Sigma}_{i,j}|) \chi_{\{ |\tilde{\Sigma}_{i,j}| < \tau,\; |\Sigma_{i,j}| \ge \tau\}} \\
    	&\le \max_{r \le \alpha p} |\tilde{\sigma}_r - \sigma_r| \cdot \max_i \sum_{j=1}^p \chi_{\{ |\Sigma_{i,j}| \ge \tau\}} + \tau \max_i \sum_{j=1}^p \chi_{\{ |\Sigma_{i,j}| \ge \tau\}} \\
    	&\le s \tau^{-q} \sqrt{\frac{c}{1-\alpha}} \max \{CK^2,\sqrt{C}K\} \sqrt{\frac{\log(p)}{np}} + s \tau^{1-q} \\
    	&\lesssim s \tau^{1-q}.
    \end{align*}
    Combining the bounds for $(I)$, $(II)$, and $(III)$ with the explicit form of $\tau$ yields the claim.
\end{proof}

\paragraph{Combining Toeplitz structure and low-rankness.} Sparsity is not the only structure that can be imposed on Toeplitz matrices. For instance, in Massive MIMO, cf.\ Section \ref{JMsec:MassiveMIMO}, low-rankness of $\boldsymbol \Sigma$ may naturally be assumed in addition to Toeplitz structure \cite{haghighatshoar2018low}. The recent works \cite{eldar2020sample,lawrence2020low} propose several algorithms to estimate low-rank Toeplitz covariance matrices from partial observations by a technique called "sparse ruler". In particular, the authors can show that the sufficient number of samples to approximate $\boldsymbol \Sigma$ scales (up to $\log$-factors) polynomial in the (effective) rank of $\boldsymbol \Sigma$.


\section{Estimation from quantized samples}
\label{JMsec:OneBitCE}

All above results assume real-valued sample vectors $\mathbf X^k$, i.e., infinite precision representation of the samples.
In applications, this assumption is hardly fulfilled. Especially in signal processing, samples are collected via sensors and, hence, need to be quantized to finitely many bits before they can be digitally transmitted and further processed. Engineers have been examining the influence of coarse quantization on correlation and covariance estimation for decades, e.g., \cite{bar2002doa,choi2016near,jacovitti1994estimation,li2017channel,roth2015covariance}. However, in contrast to classical covariance estimation from un-quantized samples, so far only asymptotic estimation guarantees have been derived in the quantized setting. To improve our understanding on the effect of quantization on covariance estimation, we analyzed two memoryless one-bit quantization schemes in our recent work \cite{dirksen2021covariance}. We call a quantizer memoryless if it quantizes each entry of $\mathbf X^k$ independently of all remaining entries. This is fundamentally different from feedback systems, e.g., $\Sigma\Delta$-quantization \cite{schreier1996delta,benedetto2006sigma}, and of particular interest for large-scale applications like Massive MIMO where the entries of $\mathbf X^k$ correspond to inputs from different antennas, cf.\ Section \ref{JMsec:MassiveMIMO}. We conclude by providing a detailed discussion of the models and results in \cite{dirksen2021covariance}.

\subsection{Sign quantization}
\label{JMsec:OneBitSign}

In the first setting, we assume to receive one-bit quantized samples
\begin{align} \label{JMeq:OneBitModel}
	\mathrm{sign}(\mathbf X^k) \in \{-1,1\}^p,
\end{align}
for $k \in [n]$, instead of $\mathbf X^k$ itself. (Recall that we apply scalar functions like sign entry-wise to vectors and matrices.) Since the quantizer $\mathrm{sign}$ is scale-invariant, i.e., $\mathrm{sign}(\mathbf z) = \mathrm{sign}(\mathbf D\mathbf z)$ for any diagonal matrix $\mathbf D\in \R^{p\times p}$ with strictly positive entries and $\mathbf z \in \R^p$, we only hope to recover the correlation matrix of the distribution, i.e., a normalized version of $\boldsymbol \Sigma$ with entries $\big[ \tfrac{\Sigma_{i,j}}{\sqrt{\Sigma_{i,i}} \sqrt{\Sigma_{j,j}}} \big]_{i,j}$. We thus assume that $\mathbf X \sim \mathcal{N}(\boldsymbol 0,\boldsymbol \Sigma)$, where $\boldsymbol \Sigma$ has ones on its diagonal.

It is common knowledge that
\begin{align} \label{JMeq:OneBitEstimator}
	\tilde{\boldsymbol \Sigma}_n = \sin \left( \frac{\pi}{2n} \sum_{k=1}^n \mathrm{sign} (\mathbf X^k) \mathrm{sign}(\mathbf X^k)^\top \right)
\end{align}
is well-suited to approximate $\boldsymbol \Sigma$ from the quantized samples, cf.\ \cite{jacovitti1994estimation}.
Note that the specific form of $\tilde{\boldsymbol \Sigma}_n$ is motivated by Grothendieck's identity (see, e.g., \cite[Lemma 3.6.6]{vershynin2018high}), also known as "arcsin-law" in the engineering literature \cite{jacovitti1994estimation,van1966spectrum}, which implies that    
\begin{align} \label{JMeq:Grothendieck}
	\boldsymbol \Gamma := \mathbb E [\mathrm{sign} (\mathbf X^k) \mathrm{sign} (\mathbf X^k)^\top ] = \frac{2}{\pi} \arcsin (\boldsymbol \Sigma)
\end{align}
if $\mathbf X \sim \mathcal N (\boldsymbol 0,\boldsymbol \Sigma)$. Applying the strong law of large numbers and the continuity of the sine function to \eqref{JMeq:OneBitEstimator} one easily obtains with \eqref{JMeq:Grothendieck} that $\tilde{\boldsymbol \Sigma}_n$ is a consistent estimator of $\boldsymbol \Sigma$. 

The two key quantities for understanding the non-asymptotic performance of $\tilde{\boldsymbol \Sigma}_n$ are $\boldsymbol \Gamma$ and
\begin{align*}
	\mathbf A := \cos(\arcsin(\boldsymbol \Sigma)) = \cos( \tfrac{\pi}{2} \boldsymbol \Gamma ).
\end{align*}
Furthermore, we define
\begin{align*}
	\sigma(\mathbf Z)^2 := \mathbf Z^2 \odot \boldsymbol \Gamma - ( \mathbf Z \odot \boldsymbol \Gamma )^2  = \frac{2}{\pi} \mathbf Z^2 \odot \arcsin(\boldsymbol \Sigma) - \frac{4}{\pi^2} \big( \mathbf Z \odot \arcsin(\boldsymbol \Sigma) \big)^2,
\end{align*}
for symmetric $\mathbf Z \in \R^{p\times p}$.

\begin{theorem}[{\cite[Theorem 1]{dirksen2021covariance}}] \label{JMthm:Operator}
	There exist constants $c_1,c_2 > 0$ such that the following holds. Let $\mathbf X \sim \mathcal N(\boldsymbol 0,\boldsymbol \Sigma)$ with $\Sigma_{i,i} = 1$, for $i \in [p]$, and $\mathbf X^1,...,\mathbf X^n \overset{\mathrm{d}}{\sim} \mathbf X$ be i.i.d.\ samples of $\mathbf X$. Let $\mathbf M \in [0,1]^{p\times p}$ be a fixed symmetric mask. Then, for all $t \ge 0$ with $n \ge c_1 \log^3(p) (\log(p) + t)$, the biased sign estimator $\tilde{\boldsymbol \Sigma}_n$ fulfills with probability at least $1 - 2e^{-c_2 t}$
	\begin{align}
		\label{JMeqn:OperatorEst}
		\begin{split}
			\| \mathbf M \odot \tilde{\boldsymbol \Sigma}_n - \mathbf M \odot  \boldsymbol \Sigma \| 
			& \lesssim  \left\| \sigma \left(\mathbf M \odot \mathbf A \right)\right\|  \sqrt{\frac{\log(p) + t}{n}} \\
			& \qquad \qquad + \left( \max\left\{ \| \mathbf M \odot \mathbf A \|, \| \mathbf M \odot \boldsymbol \Sigma \| \right\} \right) \frac{\log(p) + t}{n}.
		\end{split}
	\end{align}
\end{theorem}

The estimate in Theorem \ref{JMthm:Operator} (for convenience, we only consider the case $\mathbf M = \boldsymbol 1$ here) can be simplified \cite[Remark 3]{dirksen2021covariance} to
\begin{align*}
	\| \tilde{\boldsymbol \Sigma}_n - \boldsymbol \Sigma \| 
	\lesssim \max  \{  \| \cos(\arcsin(\boldsymbol \Sigma)) \|,\| \boldsymbol \Sigma \| \} \left( \sqrt{ \frac{\log(p) + t}{n}} + \frac{\log(p) + t}{n} \right)
\end{align*}
which is up to the additional dependence on $\cos(\arcsin(\boldsymbol \Sigma))$ comparable to the error bound in \eqref{JMeq:Estimation_Gauss} for $\hat{ \boldsymbol \Sigma }_n$. This is remarkable since $\tilde{\boldsymbol \Sigma}_n$ accesses considerably less information on the samples than $\hat{ \boldsymbol \Sigma }_n$.

Theorem \ref{JMthm:Operator} even suggests that for strongly correlated distributions of $\mathbf X$, i.e., $\boldsymbol \Sigma \approx \boldsymbol 1$, the dominant first term on the right-hand side of \eqref{JMeqn:OperatorEst} vanishes. In other words, the bound in \eqref{JMeqn:OperatorEst} predicts $\tilde{\boldsymbol \Sigma}_n$ to outperform $\hat{\boldsymbol \Sigma}_n$ if the entries of $\mathbf X$ strongly correlate. Numerical experiments confirm this counter-intuitive fact, cf.\ Figure \ref{JMfig:CorrelationFigure}. A possible explanation is that by construction $\tilde{\boldsymbol \Sigma}_n$ implicitly uses the assumption that $\boldsymbol \Sigma$ has ones on its diagonal which is not provided to $\hat{\boldsymbol \Sigma}_n$.

\begin{figure}[t]
	\centering
	\includegraphics[scale=0.49]{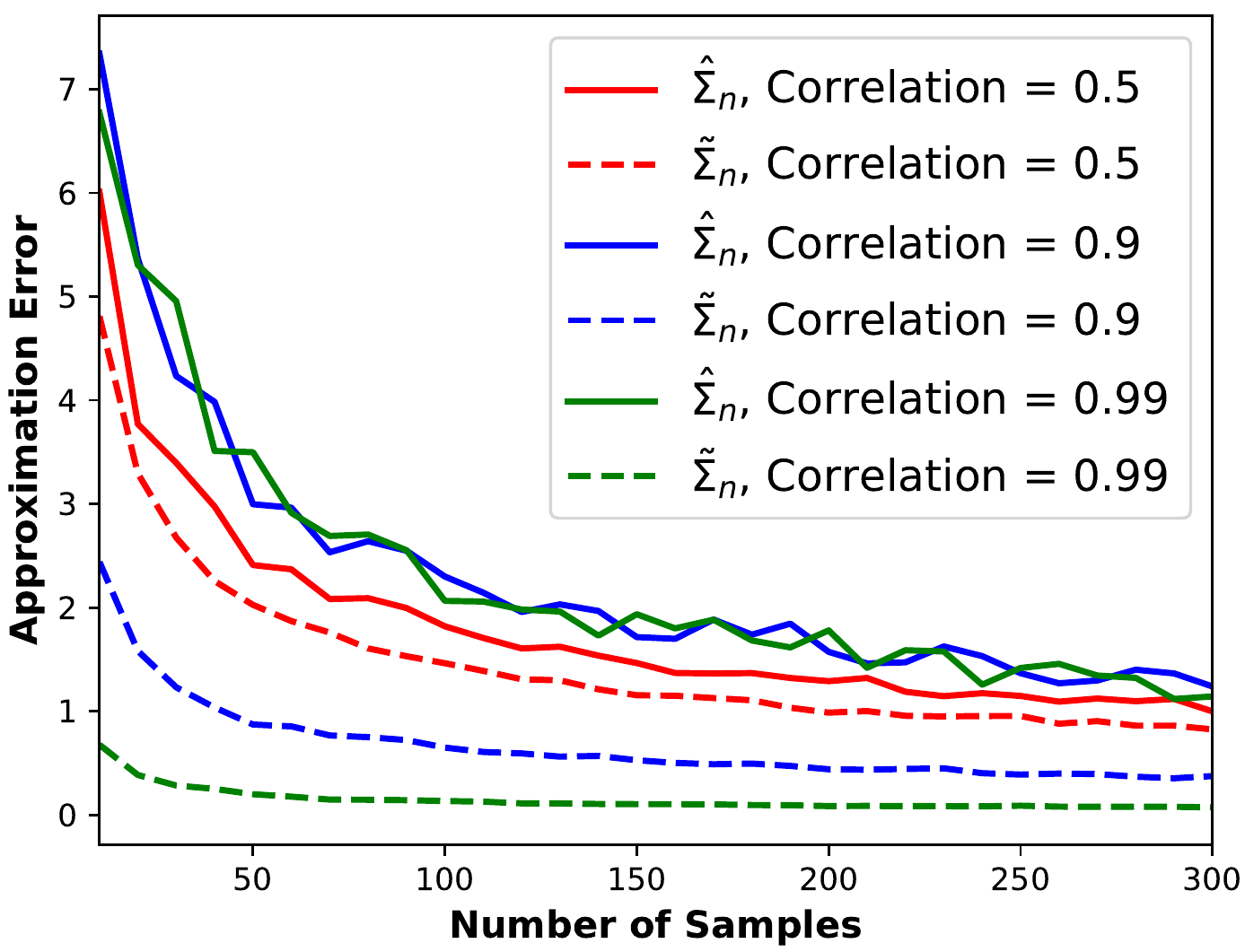}
	%
	%
	\caption{The experiment from \protect\cite{dirksen2021covariance} depicts average estimation error of $\hat{\boldsymbol \Sigma}_n$ and $\tilde{\boldsymbol \Sigma}_n$ in operator norm, for $p = 20$, $n$ varying from $10$ to $300$ and three different choices of the ground-truth $\boldsymbol \Sigma$ with ones on the diagonal and off-diagonal entries equal to $c = 0.5$, $c = 0.9$, and $c = 0.99$.}
	\label{JMfig:CorrelationFigure}       
\end{figure}

Furthermore, a corresponding lower bound on the second moment of the estimation error shows that the unconventional term $\| \sigma ( \mathbf M \odot \mathbf A ) \|$ is factual and not an artifact of the proof.

\begin{proposition}[{\cite[Proposition 15]{dirksen2021covariance}}] \label{JMprop:LB}
	There exist constants $c_1,c_2 > 0$ such that the following holds. Let $\mathbf X \sim \mathcal N(\boldsymbol 0,\boldsymbol \Sigma)$ with $\Sigma_{i,i} = 1$, for $i \in [p]$, $\mathbf X^1,...,\mathbf X^n \overset{\mathrm{d}}{\sim} \mathbf X$ are i.i.d.\ samples of $\mathbf X$, and $\mathbf M \in [0,1]^{p\times p}$ is a fixed symmetric mask. If $n \ge c_1 \log(p)$, we have that
	\begin{align*}
		\mathbb E \left[ \| \mathbf M \odot \tilde{\boldsymbol \Sigma}_n - \mathbf M \odot \boldsymbol \Sigma \|^2 \right]^{\frac{1}{2}} 
		&\gtrsim \frac{c_2}{\sqrt{n}} \left\|\sigma \left( \mathbf M \odot \mathbf A \right) \right\| 
		+ \frac{c_2}{n} \left\| \mathbf M \odot \boldsymbol \Sigma \odot \left( \boldsymbol{1} - \boldsymbol \Gamma^{\odot 2}  \right) \right\| \\
		&  + \frac{c_2}{ n } \| \sigma ( \mathbf M \odot \boldsymbol \Sigma)^2 \odot \boldsymbol \Gamma \|^\frac{1}{2} - \mathcal{O} \left( \left( \frac{\log^2(p)}{n} \right)^{\frac{3}{2}} \right).
	\end{align*}
\end{proposition}

\subsection{Dithered quantization}

The results of Section \ref{JMsec:OneBitSign} are restricted to the estimation of correlation matrices of Gaussian distributions. Both limitations stem from the chosen quantization model: first, \eqref{JMeq:OneBitModel} is blind to the re-scaling of variances and, second, Grothendieck's identity only holds for Gaussian distributions. Nevertheless, by introducing a \emph{dither} to the one-bit quantizer in \eqref{JMeq:OneBitModel} we can fully estimate the covariance matrix of general subgaussian distributions. Dithering means adding artificial random noise (with a suitably chosen distribution) to the samples before quantizing them to improve reconstruction from quantized observations, cf.\ \cite{roberts1962picture,GrN98,GrS93}. In the context of one-bit compressed sensing, the effect of dithering was recently rigorously analyzed in \cite{BFN17,DiM18a,DiM18b,JMP19,KSW16}.

To be precise, we require two bits per entry of each sample vector where each bit is dithered by an independent uniformly distributed dither, i.e., we are given
\begin{align} \label{JMeq:OneBitDithered}
	\mathrm{sign} (\mathbf X^k + \boldsymbol \tau^k), \ \mathrm{sign} (\mathbf X^k + \bar{\boldsymbol \tau}^k)^\top, \qquad k=1,\ldots,n,
\end{align}
where the dithering vectors $\boldsymbol \tau^1,\bar{\boldsymbol \tau}^1,\ldots,\boldsymbol \tau^n,\bar{\boldsymbol \tau}^n$ are independent and uniformly distributed in $[-\lambda,\lambda]^p$, with $\lambda>0$ to be specified later. From the quantized observations in \eqref{JMeq:OneBitDithered}, we construct the estimator
\begin{align} \label{JMeq:TwoBitEstimator}
	\tilde{\boldsymbol \Sigma}_n^\text{dith} = \tfrac{1}{2}\tilde{\boldsymbol \Sigma}'_n + \tfrac{1}{2}(\tilde{\boldsymbol \Sigma}'_n)^\top
\end{align}
where
\begin{align} \label{JMeq:AsymmetricEstimator}
	\tilde{\boldsymbol \Sigma}'_n = \frac{\lambda^2}{n}\sum_{k=1}^n \mathrm{sign} (\mathbf X^k + \boldsymbol \tau^k) \mathrm{sign} (\mathbf X^k + \bar{\boldsymbol \tau}^k)^\top.
\end{align}

\begin{figure}[t]
	\centering
	\includegraphics[width=7cm]{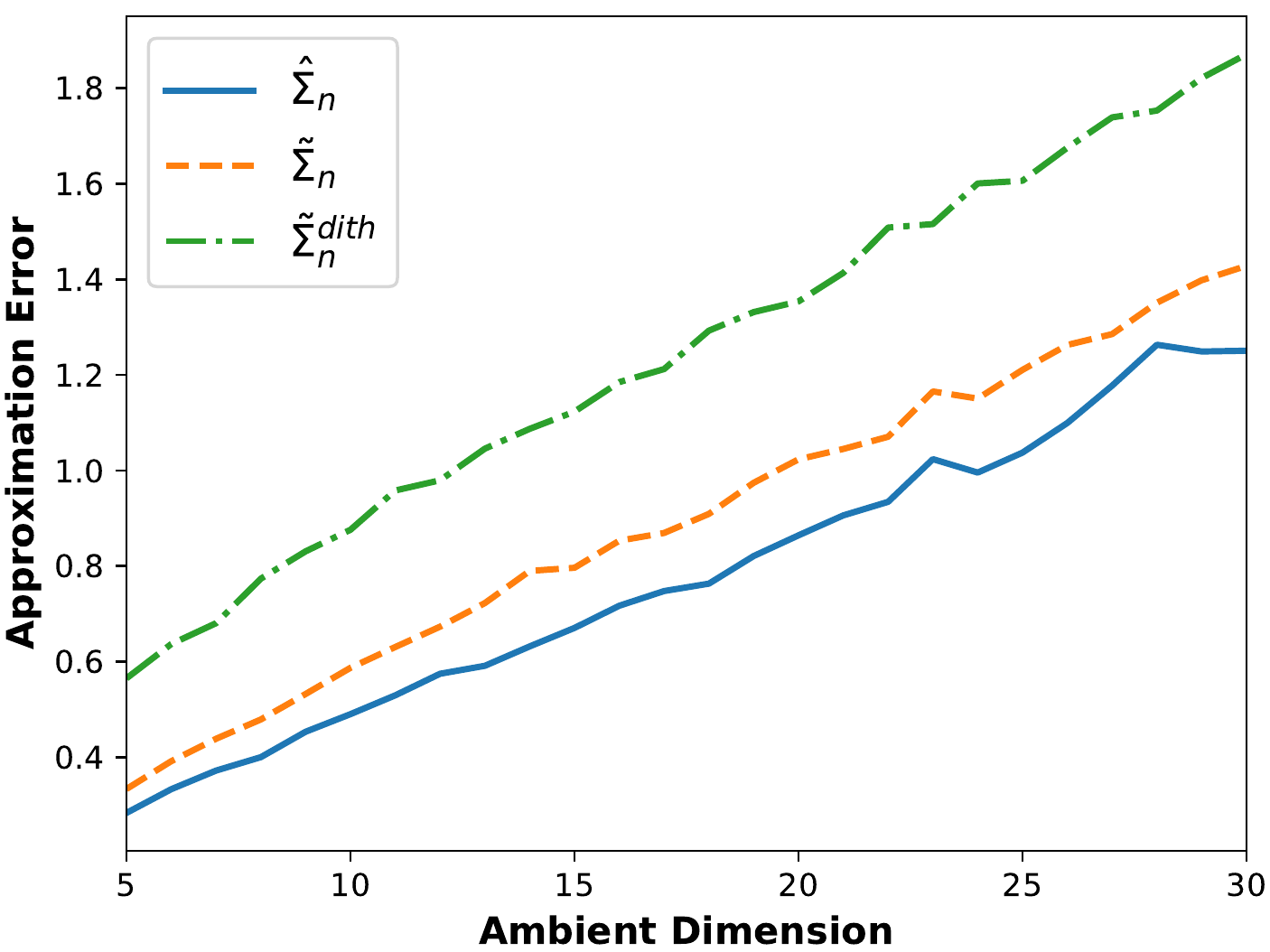}
	%
	%
	\caption{Comparison of both one-bit estimators \protect\cite{dirksen2021covariance} with $\hat{ \boldsymbol \Sigma }_n$ for $\boldsymbol \Sigma$ having ones on the diagonal. The plot depicts average estimation error in operator norm, for $n = 200$ and $p$ varying from $5$ to $30$. The dithered estimator here uses $\lambda \in (0,4 \| \boldsymbol \Sigma \|_{\infty})$ optimized via grid-search.}
	\label{JMfig:AllEstimators}       
\end{figure}

\begin{theorem} [{\cite[Theorem 4]{dirksen2021covariance}}] \label{JMthm:OperatorDitheredMask}
	Let $\mathbf X$ be a mean-zero, $K$-subgaussian vector with covariance matrix $\mathbb E [ \mathbf X \mathbf X^\top ] = \boldsymbol \Sigma$. Let $\mathbf X^1,...,\mathbf X^n \overset{\mathrm{d}}{\sim} \mathbf X$ be i.i.d.\ samples of $\mathbf X$. Let $\mathbf M \in [0,1]^{p\times p}$ be a fixed symmetric mask. If $\lambda^2 \gtrsim \log(n) \| \boldsymbol \Sigma \|_{\infty}$, then with probability at least $1-e^{-t}$,  
	\begin{align*}
		\| \mathbf M \odot \tilde{\boldsymbol \Sigma}_n^\text{dith} &- \mathbf M \odot \boldsymbol \Sigma \| \\
		&\lesssim \|\mathbf M\|_{1\to 2}(\lambda\|\boldsymbol \Sigma \|^{1/2}+\lambda^2)\sqrt{\frac{\log(p)+t}{n}} + \lambda^2\|\mathbf M\| \frac{\log(p)+t}{n}.
	\end{align*}
	In particular, if $\lambda^2 \approx \log(n) \|\boldsymbol \Sigma \|_{\infty}$, we have
	\begin{align} \label{JMeqn:OperatorDitheredMask}
	\begin{split}
	    &\| \mathbf M \odot  \tilde{\boldsymbol \Sigma}_n^\text{dith} - \mathbf M \odot \boldsymbol \Sigma \| \\
	    &\lesssim \log(n) \| \mathbf M \|_{1\to 2} \sqrt{\frac{\|\boldsymbol \Sigma \| \ \|\boldsymbol \Sigma \|_{\infty}(\log(p)+t)}{n}} + \log(n)\|\mathbf M\|\|\boldsymbol \Sigma \|_{\infty}\frac{\log(p)+t}{n}.
    \end{split}
	\end{align}
\end{theorem}

The error bound \eqref{JMeqn:OperatorDitheredMask} coincides (up to different logarithmic factors) with the best known estimate for the masked sample covariance matrix in Theorem \ref{JMthm:Chen}, even though the sample covariance matrix requires direct access to the samples $\mathbf X^k$, cf. Figure \ref{JMfig:AllEstimators}. This performance, however, heavily depends on the choice of $\lambda$, as Figure \ref{JMfig:TuningLambda} shows. Furthermore, it should be mentioned that there are cases when the performance of the dithered estimator is significantly worse than the performance of the sample covariance matrix. Let us consider for simplicity the case $\mathbf M=\boldsymbol{1}$). If the samples $\mathbf X^k$ are Gaussian, then \cite{koltchinskii2014concentration} shows that
\begin{align*}
	\mathbb E [ \| \hat{\boldsymbol \Sigma}_n - \boldsymbol \Sigma \| ] \simeq \sqrt{\frac{\|\boldsymbol \Sigma \| \mathrm{Tr}(\boldsymbol \Sigma)}{n}} + \frac{\mathrm{Tr}(\boldsymbol \Sigma)}{n},
\end{align*}
whereas \eqref{JMeqn:OperatorDitheredMask} yields 
\begin{align*}
	\mathbb E [ \| \tilde{\boldsymbol \Sigma}_n^\text{dith} - \boldsymbol \Sigma \| ] 
	\lesssim 
	\log(n) \sqrt{\frac{p\|\boldsymbol \Sigma \| \ \|\boldsymbol \Sigma \|_{\infty}\log(p)}{n}} + \log(n)\frac{p\|\boldsymbol \Sigma \|_{\infty}\log(p)}{n}
\end{align*}
via tail integration. Since $\mathrm{Tr}(\boldsymbol \Sigma) \leq p\|\boldsymbol \Sigma \|_{\infty}$, the second estimate is worse in general. Numerical experiments in \cite{dirksen2021covariance} have shown that this difference is not an artifact of proof. Simply put, $\hat{\boldsymbol \Sigma}_n$ and $\tilde{\boldsymbol \Sigma}_n^\text{dith}$ perform similarly if $\boldsymbol \Sigma$ has a constant diagonal, whereas $\hat{\boldsymbol \Sigma}_n$ performs significantly better whenever $\mathrm{Tr}(\boldsymbol \Sigma) \ll p\|\boldsymbol \Sigma \|_{\infty}$.

\begin{figure}[t]
	\centering
	\includegraphics[width=7.2cm]{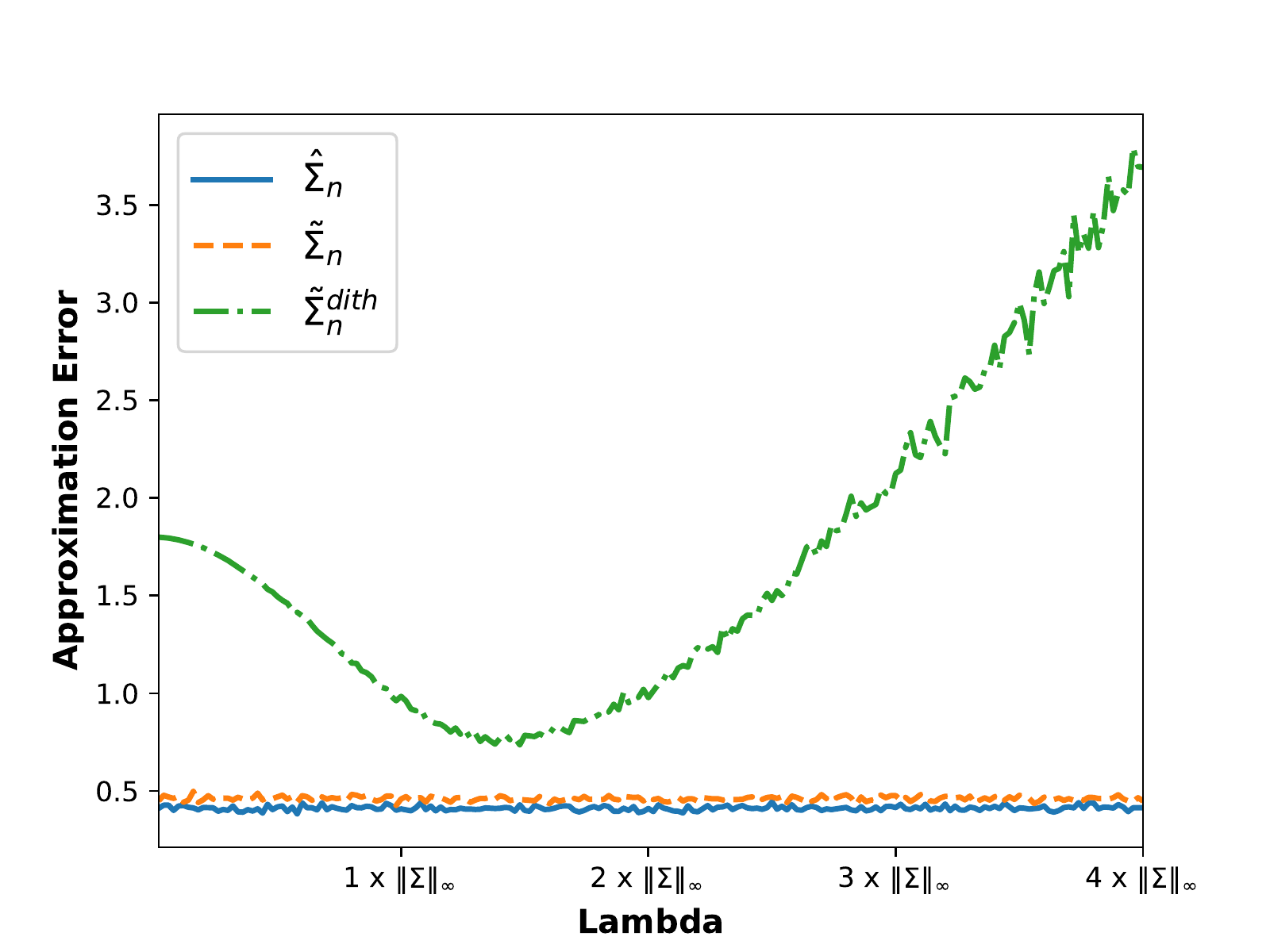}
	%
	%
	\caption{Experiment on the influence of $\lambda$ \protect\cite{dirksen2021covariance} on the reconstruction performance of $\tilde{\boldsymbol \Sigma}_n^\text{dith}$. The plot depicts average estimation error in operator norm, for $n = 200$, $p=5$, and $\lambda$ varying from $0$ to $4 \| \boldsymbol \Sigma \|_{\infty}$. Though not affected by changes in $\lambda$, sample covariance matrix, and un-dithered estimator are given for reference.}
	\label{JMfig:TuningLambda}       
\end{figure}

Theorem \ref{JMthm:OperatorDitheredMask} can be extended to heavier-tailed random vectors. This, however, requires a larger choice of $\lambda$ and thus more samples to reach the same error. For a sub-exponential random vector $\mathbf X$, one would already need $\lambda^2 \gtrsim \log(n)^2 \cdot \max_{i \in [p]} \| X_i \|_{\psi_1}^2$.
The dependence of $\lambda$ on $n$, both in the latter statement and Theorem \ref{JMthm:OperatorDitheredMask} can be observed in numerical experiments \cite{dirksen2021covariance} as well.

Let us finally mention that the quantized estimators in \eqref{JMeq:OneBitEstimator} and \eqref{JMeq:TwoBitEstimator} are not necessarily positive semi-definite as one expects from covariance matrices. In applications one would thus replace both estimators by their projection onto the cone of positive semi-definite matrices, which is efficiently computed via the singular value decomposition \cite[Section 8.1.1]{boyd2004convex}. The obtained estimates also apply to the projected estimators since convex projections are $1$-Lipschitz.

\section*{Acknowledgements}

All authors acknowledge  funding  by the  Deutsche  Forschungsgemeinschaft  (DFG,  German  Research  Foundation)  through  the  project CoCoMIMO  funded  within  the  priority  program  SPP  1798 Compressed  Sensing  in  Information Processing(COSIP).

\bibliographystyle{IEEEtranS}                 
\bibliography{bibfileJM}                